\documentclass[12pt,a4paper]{amsart}

\usepackage{graphicx} % Include figure files
\usepackage{dcolumn} % Align table columns on decimal point
\usepackage{bm} % bold math
\usepackage{color}
\usepackage[T1]{fontenc}
\usepackage{mathptmx}

\usepackage{amsthm}
\usepackage{amsfonts}
\usepackage{fancyhdr}
\usepackage{mathrsfs}
\usepackage{latexsym}
\usepackage{graphicx}
\usepackage{amsmath}
\usepackage{epstopdf}
\usepackage{enumitem}

\usepackage{hyperref}
\newtheorem{definition}{Definition}[section]
\newtheorem{lemma}{Lemma}[section]

\newtheorem{corollary}{Corollary}[section]

\newcommand{\ma}{\begin{pmatrix}}
\newcommand{\am}{\end{pmatrix}}

\newcommand{\be}[1]{\begin{equation} \label{#1}}
\newcommand{\ee}{\end{equation}}
\newcommand{\pdpd}[2]{\frac{\partial #1}{\partial #2}}
\newcommand{\dede}[2]{\frac{{\operatorname{d}} #1}{{\operatorname{d}} #2}}

\newcommand{\ii}{\mathrm{i}}
\newcommand{\R}{\mathbb{R}}
\newcommand{\C}{\mathbb{C}}

\renewcommand{\Im}{\operatorname{Im}}
\renewcommand{\Re}{\operatorname{Re}}

\def \iy{\infty}
\def \sm {\setminus}

\def \vphi {\varphi}

\newcommand{\gM}{\mathfrak{M}}

\def \bF{{\bf F}}
 \def\bH{{\bf H}} 
\def \bM{{\bf M}}

\def \bG{{\bf G}}

\def \bA{{\bf A}}
\def \bB{{\bf B}}
\def \bC{{\bf C}}

\def \bm1{{\bf m}_1}

\def \cK{\mathcal{K}}
\def \cR{\mathcal{R}}

\def \cP{{\mathcal P}}
\def \cS{{\mathcal S}}
\def \cN{{\mathcal N}}

\def\lb{\label}

\def\mR{{\mathfrak{R}}}

\def \mB{{\mathscr B}}

\def \vt{\vartheta}
\def \vp{\varphi}

\def \hm{\mu}
\def \s{\sigma}

\def \qq{\quad}

\def\[{\begin{equation}}
\def\]{\end{equation}}

\def\qq{\quad}

\theoremstyle{plain}
\newtheorem{theorem}{Theorem}[section]
\theoremstyle{definition} 
\newtheorem{remark}[theorem]{Remark}

\begin{document}

\markboth{M. V. de Hoop, and A. Iantchenko}{Wavenumber resonances for the Rayleigh system}

\title{Analysis of wavenumber resonances\\ for the Rayleigh system in a half space}

\author{Maarten V. de Hoop}
\address{Simons Chair in Computational and Applied Mathematics and Earth Science, Rice University, Houston, TX 77005}
\email{mdehoop@rice.edu}
\author[Alexei Iantchenko]{Alexei Iantchenko${}^1$%\footnotemark[1]
}
\footnotetext[1]{Corresponding author}
\address{Department of Materials Science and Applied Mathematics, Faculty of Technology and Society, Malm\"{o} University, Malm\"{o}, Sweden}
\email{ai@mau.se}

\begin{abstract}
We present a comprehensive analysis of wavenumber resonances or leaky modes associated with the Rayleigh operator in a half space containing a heterogeneous slab, being motivated by seismology. To this end, we introduce Jost solutions on an appropriate Riemann surface, a boundary matrix and a reflection matrix in analogy to the studies of scattering resonances associated with the Schr\"{o}dinger operator. We analyze their analytic properties and characterize the distribution of these wavenumber resonances. Furthermore, we show that the resonances appear as poles of the meromorphic continuation of the resolvent to the nonphysical sheets of the mentioned Riemann surface as expected.
\end{abstract}


%\date{September 30, 2022}


\maketitle

\section{Introduction}

Scattering resonances are associated with oscillations and rates of decay of solutions of a wave equation, and exist in certain settings. These settings determine a Riemann surface. Scattering resonances are the poles of a meromorphic continuation of the resolvent of the time-Fourier-transformed wave operator, in frequency $\omega$, say, to the unphysical sheets of the Riemann surface, and appear as Breit-Wigner bumps on the graph of the scattering phase or Green's function. The mathematical theory of scattering resonances is relatively new but already well developed. A modern review is given in a recent book by Dyatlov and Zworski \cite{DyatlovZworski2022}. In the analysis of inverse problems for the Schr\"{o}dinger operator, resonance frequencies have been used as the only data. It is indeed the inverse problem for the Rayleigh operator that motivated the present study.

In elasticity, and specifically seismology, for the Rayleigh operator, the notion of resonances is slightly different from the one of resonances associated with the Schr\"{o}dinger operator. We refer to these as wavenumber resonances while seismologists prefer the term leaky modes. The occurrence of leaky modes has been noted more than 60 years ago. Yet, progress in detecting and exploiting them in studying Earth's interior -- rather than exploiting Love and Rayleigh modes or surfaces waves \cite{Rayleigh1885} -- has been very limited, in part due to the absence of their comprehensive analysis. We present such an analysis here in dimension three, which opens the way for studying the corresponding inverse problem. More precisely, we study a wavenumber resonances problem associated with the Rayleigh system in the isotropic case, at a fixed frequency $\omega$, on a flat elastic half space. The coefficients in the system depend on the boundary normal coordinate and vary in a slab of finite thickness beneath a traction-free surface. We refer to the half space below the slab as the lower half space. Complications in the analysis arise from the fact that isotropic elasticity supports two distinct wave speeds.

Assuming isotropy when the stiffness tensor is determined by two Lam\'{e} parameters, we introduce Jost solutions, which appear as complex plane waves in the lower half space, as \textit{P}- and \textit{S}-polarized displacements and dependent on a spectral parameter, $\xi$ say, originating as the norm of the dual to the boundary or surface coordinates. The natural object for studying resonances is the boundary matrix representing tractions induced by the Jost solutions. The determinant of the boundary matrix is known as the Rayleigh determinant. The Rayleigh determinant as a function of $\xi$ has as its complex roots the resonance frequencies. The boundary matrix determines the Jost function or spectral data. In previous work \cite{MdHoopIantchenko2022}, we proved that the Lam\'{e} parameters in the slab can be uniquely recovered from this Jost function assuming that the density is known. In a follow-up paper we will present the corresponding inverse wavenumber resonances problem. A reflection matrix for the Rayleigh system, which can be expressed in terms of the boundary matrix, can be identified as the analogue of the scattering matrix in the inverse scattering problem for the Schr\"{o}dinger equation.

The Riemann surface is introduced and based on the Jost solutions restricted to the lower half space with constant Lam\'{e} parameters.
We analyze the analytic properties of the Jost solutions within the slab, in $\xi$, on this Riemann surface and study their asymptotic behaviors. Furthermore, we study the properties of the reflection and boundary matrices through identities, and determine how these matrices are related. We note that the reflection matrix does not determine the boundary matrix~\footnote{In comparison, in the case of the Schr\"{o}dinger operator, the scattering matrix determines the Jost function.}. We then introduce a function, $F$ say, as a product of the Rayleigh determinant and its three ``conjugates'' associated to three unphysical sheets of the mentioned Riemann surface (exact definition is given in (\ref{functionF})) and prove that $F$ is an entire function of exponential type belonging to a Cartwright class. The zeros of $F$ are projections of the wavenumber resonances on the complex plane. Then, using the known properties of zeros of Cartwright class functions, we establish the distribution of the wavenumber resonances.

We let $\cK$ denote the cut plane being the natural projection of the mentioned Riemann surface, $\mR$, defined in Section~\ref{s-R}. We let $\cN(r,F)$ denote the number of zeros of the mentioned function $F(-\ii z)$, $-\ii z \in \cK$, having modulus  $\leq r$, each zero being counted according to its multiplicity. With the Cartwright character of $F(-\ii z)$ proved in Section~\ref{ss-A} and general properties of the zeros of an entire function from a Cartwright class summarized in Section~\ref{ss-C}, we obtain the following main results. First, the wavenumber bound states and resonances, $\xi_n = -\ii z_n$, satisfy 
\begin{equation} \label{xisumcond}
   \sum \frac{|\Im \xi_n |}{|\xi_n|^2} < \infty .
\end{equation}
Second,
\[\label{Weylas}
   \cN(r,F) = \frac{16 H \, r}{\pi} (1 + o(1)) ,\qq r \rightarrow \infty ,
\]
where $H$ signifies the thickness of the slab. This can be compared with the analogous result for the Schr{\"o}dinger operator with compactly supported potential on the half line \cite[Theorem~2.1]{Korotyaev2004}. Third,  under the condition that Lam{\'e} moduli $\mu, \lambda \in C^N$ for some $N$ sufficiently large, each wavenumber resonance, $\xi_n$, $n = 1,2,\ldots$, with $\Re\xi_n < 0$ satisfies
\begin{equation} \label{3.14abisbis}
   \left|\xi_n^2 \left(\xi_n^{12} - c_2 \xi_n^{10} + c_4 \xi_n^{8}
   - c_6 \xi_n^{6} + c_8 \xi_n^4 - c_{10} \xi_n^2 + c_{12}\right)\right|
   \leq C_1 e^{-8 H \Re\xi_n}
\end{equation}
for some constants $c_{2j} \in \R$, $j = 1,\ldots 6$. As a consequence, for any $A > 0,$ there are only finitely many resonances in the region
\[\label{forbreg}
   0 > \Re\xi \geq -A - \frac{7}{4H} \log|\Im\xi| ,
\] 
specifying the forbidden domain for the Rayleigh wavenumber resonances. It is well known that for \\ Schr{\"o}dinger and Dirac operators with compactly supported potential, the scattering resonances in frequency ($\omega$) lie below a logarithmic curve in $\C_-$, see \cite[Corollary~2.3]{Korotyaev2004}, \cite[Theorem~1.3]{IantchenkoKorotyaev2014a} and \cite[Theorem~2.7]{IantchenkoKorotyaev2014b}. Bound (\ref{3.14abisbis}) shows that wavenumber resonances (identifying $\xi$ with $-\ii \omega$) with $\Re\xi_n < 0$ lie to the left of the logarithmic curve for $\Re\xi < 0$.

Furthermore, we derive an explicit formula for the kernel of the resolvent (Green's function) of the Rayleigh operator in the half space. Using the analytic properties of the Jost solutions on the Riemann surface, we define the analytic continuation of the resolvent from the physical sheet to the whole Riemann surface, and observe that the poles of the analytic continuation of the resolvent coincide with the zeros of the Rayleigh determinant, that is, the wavenumber resonances. 

For studies of leaky modes in seismology we refer to \cite{Phinney1961, Phinney1961b, Phinney1961c, Gupta1970, Chapman1972, Watson1972, SchroderScott2001, HarrisAchenbach2002}. Pilant \cite{Pilant1972} and Haddon \cite{Haddon1984, Haddon1986, Haddon1987} presented a complex frequency, complex wavenumber analysis of leaky modes. Lodge, Steblov and Gubbins \cite{Lodgeetal1999}, identified the fundamental (PL) leaky mode. Wu and Chen \cite{WuChen2017} carried out computations of leaky modes for anomalous, layered models. Li, Shi, Ren and Chen \cite{Lietal2021} succeeded in extracting multiple leaky mode dispersion observations from ambient noise cross-correlation data.

Concerning prior work, we mention the results  on distribution of scattering resonances on the real line by Zworski in \cite{Zworski1987}. Korotyaev \cite{Korotyaev2004} established the analytic properties of Jost solutions and distribution of resonances for one-dimensional scalar Schr{\"o}dinger operators on the half line with Dirichlet boundary condition, and on the line  \cite{Korotyaev2005}, leading to complete solution of inverse resonance problem with characterization. These problems are similar to the corresponding problem for Love wavenumber resonances and substantially simpler. For the Schr{\"o}dinger operator with compactly supported potential on a half line with Dirichlet boundary condition, the Jost function is an entire scalar function of frequency ($\omega$) on the complex plane and belongs to a Cartwright class. Cohen and Kappeler \cite{CohenKappeler1985} and Christiansen \cite{Christiansen2005} analyzed the case of steplike potentials.

For an un-formal  introduction to semiclassical inverse spectral and resonance problems in seismology  see the lectures given by the second author \cite{Iantchenko2022}, where in the last two chapters the complexity of the spectral and resonance problems for elastics medium is explained in more details.
\section{Rayleigh system}
\label{sec:2}

We let $Z$ be the boundary normal coordinate, $Z \in \mathbb{R}_{\le 0}$, and $\xi$ be the dual to the coordinates in the boundary of a half space. We consider the operator associated with Rayleigh waves in isotropic elastic media \cite{dHINZ},
\begin{equation} \label{Hamiltonian_Rayleigh}
   H_0(|\xi|)
     \left(\!\!\begin{array}{c} \varphi_1 \\[0.4cm] \varphi_3
                            \end{array}\!\!\right)
   = \left(\!\!\begin{array}{c} \displaystyle
     -\frac{\partial}{\partial Z} \left(\hat{\mu}
     \frac{\partial\varphi_1}{\partial Z}\right)
     - \ii |\xi| \left(\frac{\partial}{\partial Z}(\hat{\mu} \varphi_3)
     + \hat{\lambda} \frac{\partial}{\partial Z}\varphi_3\right)
     + (\hat{\lambda} + 2 \hat{\mu}) |\xi|^2 \varphi_1
   \\[0.25cm] \displaystyle
   -\frac{\partial}{\partial Z} \left((\hat{\lambda} + 2 \hat{\mu})
   \frac{\partial\varphi_3}{\partial Z}\right)
   - \ii |\xi| \left(\frac{\partial}{\partial Z}(\hat{\lambda} \varphi_1)
   + \hat{\mu} \frac{\partial}{\partial Z} \varphi_1\right)
   + \hat{\mu} |\xi|^2 \varphi_3 \end{array}\right) ,
\end{equation}
where
$$
   \hat{\lambda} = \frac{\lambda}{\rho} ,\qq
   \hat{\mu} = \frac{\mu}{\rho}
$$
with $\lambda$, $\mu$ denoting the Lam\'{e} parameters and $\rho$ the density of mass. In this equation, $H_0$ should be viewed as an operator-valued principal symbol. We assume that the parameters only depend on $Z$. We will use the notation $$\varphi = \ma \varphi_1 \\[0.4cm] \varphi_3 \am
.$$ We introduce the system of equations,
\begin{equation} \label{Rayleighsystem}
   (H_0(|\xi|) - \omega^2 I) \varphi = 0 ,\quad Z < 0 ,
\end{equation}
supplemented with the Neumann or traction-free boundary conditions at $Z = 0$,
\begin{align}
  a(\varphi) :=&\ \ii \hat{\lambda}(0^- |\xi| \vphi_1(0^-)
  + (\hat{\lambda}(0^-) + 2\hat{\mu}(0^-)) \pdpd{\vphi_3}{Z}(0^-)
  = 0 ,
\label{Rayleighboundary1}\\
  b(\varphi) :=&\ \ii |\xi|\hat{\mu}(0^-) \vphi_3(0^-)
  + \hat{\mu}\pdpd{\vphi_1}{Z}(0^-) = 0 .
\label{Rayleighboundary2}
\end{align}
We set $\rho \equiv 1$ and simplify the notation,
$$
  \lambda = \hat{\lambda} ,\ \mu = \hat{\mu} .
$$
From now on, we will use $\xi$ to denote both $|\xi| \in \R_+$ and its values in $\C$ following analytic continuation. Then we write $H_0 = H_0(\xi)$. We note that $H_0(\xi)$ only corresponds to the physical system for $\xi \in \R_+$. Starting from the elastic wave equation in dimension three, the equation for $\varphi_2$ decouples and describes Love waves.

\medskip\medskip

\noindent
We also consider the extension of $H_0(\xi)$ to $Z \in \R$. We write this extension as $\mathfrak{H}_0 = \mathfrak{H}_0(\xi)$, $\xi \in \C$; the extension ignores the boundary conditions. It is introduced by evenly extending its parameters, $\lambda$ and $\mu$, to $Z > 0$. 

\medskip\medskip

\noindent
We consider the case of an inhomogeneous isotropic elastic slab of thickness $H$ bonded to a homogeneous isotropic elastic half space with Lam\'e parameters $\lambda_0$ and $\mu_0$. We assume that the layer's Lam\'e parameters, $\lambda$ and $\mu$, are $C^3((-\infty,0])$ (three times continuously differentiable on $(-\infty,0]$) and are the constants $\lambda_0$ and $\mu_0$ for $Z<-H.$
 This is needed in the later analysis based on the Markushevich transform. We assume that
$$
   \mu \ge \alpha_0 > 0 ,\quad 2 \mu + 3 \lambda \ge \beta_0 > 0 ,
$$
signifying the \textit{strong ellipticity} condition \cite{Chen1991} as this appears in the existence and uniqueness of solutions of the boundary value problem for time-harmonic elastic waves.

To obtain a real-valued form of the Rayleigh system (\ref{Rayleighsystem}), we make the substitution following \cite{Stickler1986},
$$
   \psi_1 = \varphi_1 ,\qq  \psi_2 = -\ii \vp_3 ,\qq
   \text{writing}\qq \psi = \ma \psi_1 \\[0.4cm] \psi_2 \am .
$$
Then
\begin{equation} \lb{(1a)}
   (\hat{H}_0(\xi) - \omega^2 I) \psi
   := (P \psi')' + \xi (N \psi'
       - \left(N^{\rm T} \psi\right)')
       + \left(\omega^2 I - \xi^2 M\right) \psi = 0 ,
\end{equation}
where
\begin{equation}\lb{(1b)}
   P = \ma \mu & 0 \\ 0 & \lambda+2\mu \am ,\qq
   M = \ma \lambda+2\mu & 0 \\ 0 & \mu\am ,\qq
   N = \ma 0 & -\lambda \\ \mu & 0 \am
\end{equation}
and primes indicate derivatives with respect to $Z$. By the mentioned extension, we then introduce $\hat{\mathfrak{H}}_0(\xi)$.

\begin{remark}\label{r_ext}
Due to the form of (\ref{(1a)}) with $\hat{H}_0(\xi)$ replaced by $\hat{\mathfrak{H}}_0(\xi)$, it follows that if $\psi(Z,\xi)$ is a solution, then $\widetilde{\psi}(Z,\xi) := \psi(-Z,-\xi)$ is also solution. We note that this property is artificial and does not reflect the physics of the original elastic system from which the Rayleigh system is deduced.
\end{remark}

\noindent
The traction at the boundary takes the form
\begin{align}
   \hat{a}(\psi) &:= \ii \left(\lambda(0^-) \xi \psi_1(0^-)
        + (\lambda(0^-) + 2\mu(0^-)) \pdpd{\psi_2}{Z}(0^-)\right) ,\label{tract1}
\\
   \hat{b}(\psi) &:= -\xi \mu(0^-) \psi_2(0^-) + \mu(0^-)
                    \pdpd{\psi_1}{Z}(0^-) .\label{tract2}
\end{align}

\subsection*{Jost solutions in the lower half space}

The polarized Jost solutions of the Rayleigh system $(\hat{H}_0(\xi) - \omega^2) \psi = 0$, where $\hat{H}_0(\xi)$ is defined  in (\ref{(1a)}), in the lower half space, $Z < -H$, are given by
\begin{align}
   &\psi_{P,0}^\pm = \ma \psi_{P,0;1}^\pm \\ \psi_{P,0;2}^\pm \am
               = \ma -\xi \\ \pm \ii q_P\am e^{\pm \ii Z q_P} ,
\\
   &\psi_{S,0}^\pm = \ma \psi_{S,0;1}^\pm \\ \psi_{S,0;2}^\pm \am
               = \ma \pm \ii q_S  \\ - \xi  \am e^{\pm \ii Z q_S} .
\end{align}
Here,
\[\label{eq:quasimomenta}
   q_P(\xi) = \sqrt{\frac{\omega^2}{\sigma_0} - \xi^2}
\quad\text{and}\quad
   q_S(\xi) = \sqrt{\frac{\omega^2}{\hm_0} - \xi^2} ,
\]
in which $\sigma_0 := \lambda_0 + 2 \hm_0$. One refers to $q_P$ and $q_S$ as quasimomenta.

\begin{definition}\lb{Jostsol} We define Jost solutions of the Rayleigh equation (\ref{(1a)}), $$\Psi^{\pm} = [\psi_P^{\pm}\,\,\psi_S^{\pm}],$$
by matching $\Psi^{\pm}$ with $\Psi_0^{\pm}$ for $Z \le -H$, effectively imposing a radiation condition.
\end{definition}

\noindent
This definition uniquely determines the Jost solutions. We note that their boundary values will not yield a vanishing traction. 

From the boundary values of the ``downgoing'' Jost solutions, we form the \textit{boundary matrix}, $\mB$, given by
\begin{equation}
   \mB(\xi) = \ma \hat{a}(\psi_P^-)(\xi) & \hat{a}(\psi_S^-)(\xi) \\
             \hat{b}(\psi_P^-)(\xi) & \hat{b}(\psi_S^-)(\xi) \am
\end{equation}
consisting of the associated boundary tractions, and
\begin{equation} \label{eq:Rdet}
   \Delta(\xi) := \det\mB(\xi)
\end{equation}
is called the \textit{Rayleigh determinant}. The zeros of $\Delta(\xi)$ on the Riemann surface, $\cR$, defined in the next section are the \textit{wavenumber resonances}. We note that wavenumber resonances on the physical sheet, $\cR_{++}$ introduced below, correspond to the (physical) bound states.

The Jost solutions can be extended to $Z > 0$ as solutions of $(\hat{\mathfrak{H}}_0(\xi) - \omega^2 I) \psi = 0$ noting that the Lam\'{e} parameters will not be differentiable at $Z = 0$. However, the later analysis based on the Markushevich transform will not require this extension and, hence, this is not an issue.

\section{Branch cuts and Riemann surface}\label{s-R}

The Riemann surface playing a role in analyzing the Rayleigh system is determined by $q_S$ and $q_P$ in (\ref{eq:quasimomenta}). We denote by $\sqrt{z}$ the principal branch of the square root that is positive for $z>0$ and with the cut along the negative real axis.  We write $\xi \in\C$ for analytic continuation of $|\xi| \in \R_+$. Let $\omega>0$ be fixed. We define $q_S(\xi)$ by choosing the branch of $q_S(\xi) = \sqrt{\frac{\omega^2}{\hm_0} - \xi^2} = \ii \sqrt{\xi^2 - \frac{\omega^2}{\hm_0}}$ with
$$
   q_S(\xi) \in \ii \R_+\ \text{for real}\
   \xi > \frac{\omega}{\sqrt{\hm_0}}\qq\text{and}\qq
   q_S(\xi) \in \ii \R_-\ \text{for real}\
         \xi < - \frac{\omega}{\sqrt{\hm_0}} .
$$
Then 
$$
   \Im q_S(\xi) > 0\ \text{for}\
   \Re \xi > \frac{\omega}{\sqrt{\hm_0}}\qq\text{and}\qq
   \Im q_S(\xi) < 0\ \text{for}\ \Re \xi < -\frac{\omega}{\sqrt{\hm_0}} .
$$
We note that $\Im q_S(\xi) = 0$ for $\xi \in
\left[-\frac{\omega}{\sqrt{\hm_0}},\frac{\omega}{\sqrt{\hm_0}}\right]
\cup \ii \R$. We let
$$
   \cK_S := \C \setminus
   \left(\left[-\frac{\omega}{\sqrt{\hm_0}},
     \frac{\omega}{\sqrt{\hm_0}}\right] \cup \ii \R\right) .
$$
Then the map $q_S :\ \cK_S \to \cK_S ,\ \xi \to
\sqrt{\frac{\omega^2}{\hm_0} - \xi^2}$ is conformal and has the
asymptotic expansion,
\begin{equation} \label{limitS}
   q_S(\xi) = \ii \xi - \frac{\ii \omega^2}{2\hm_0\xi}
    + {\mathcal O}\left(\frac{1}{|\xi|^2}\right)\qq \text{as}\qq
    |\xi| \to \infty .
\end{equation}
We have
\begin{align}
   &q_S\left(\left[-\frac{\omega}{\sqrt{\hm_0}},
     \frac{\omega}{\sqrt{\hm_0}}\right]\right)
   = \left[-\frac{\omega}{\sqrt{\hm_0}},
     \frac{\omega}{\sqrt{\hm_0}}\right] ,
\label{slitmap} \\
   &q_S\left(\ii \R\right)
   = \left(-\infty,-\frac{\omega}{\sqrt{\hm_0}}\right] \cup
      \left[\frac{\omega}{\sqrt{\hm_0}},\infty\right) .
\label{comslitmap}
\end{align}
More precisely,
$$
   q_S(\ii \R_\pm) = \R_\mp \setminus
   \left(-\frac{\omega}{\sqrt{\hm_0}},
   \frac{\omega}{\sqrt{\hm_0}}\right) ,\qq
   q_S\left(\R_\pm \setminus
   \left(-\frac{\omega}{\sqrt{\hm_0}},
   \frac{\omega}{\sqrt{\hm_0}}\right)\right) = \ii \R_\pm .
$$
In particular, we have
\[\label{imq}
   \pm \Im(q_S(\xi)) >  0\qq \text{iff}\qq
   \xi \in \cK_{S,\pm} := \left\{\xi \in \C \setminus
   \left[-\frac{\omega}{\sqrt{\hm_0}},
     \frac{\omega}{\sqrt{\hm_0}}\right]\ :\ \pm\Re\xi > 0 \right\} .
\]
The Riemann surface for $q_S(\xi)$ is obtained by joining the upper
and lower rims of two copies of $$\C \sm
\left[\left(-\infty,-\frac{\omega}{\sqrt{\hm_0}}\right] \cup
\left[\frac{\omega}{\sqrt{\hm_0}},+\infty\right)\right]$$ cut along 
$\left(-\infty,-\frac{\omega}{\sqrt{\hm_0}}\right] \cup
      \left[\frac{\omega}{\sqrt{\hm_0}},+\infty\right)$ in the usual
        (crosswise) way.

Instead of this two-sheeted Riemann surface, it is more convenient to
work on the cut plane $\cK_S$ and half planes $\cK_{S,\pm}$ such that
$q_S\left(\cK_{S,\pm}\right) = \C_\pm := \{z \in \C\ :\ \pm \Im z >
0\}$. Let $\mathfrak{g}_+$ denote the upper rim of the cut
$\left[-\frac{\omega}{\sqrt{\hm_0}},\frac{\omega}{\sqrt{\hm_0}}\right]
\cup \ii \R$. The ``upper'' (physical) sheet for $q_S$ corresponds to
\begin{equation} \label{cKS}
   \left\{\xi \in \cK_S\ :\ \Re \xi > 0 \right\} \cup \mathfrak{g}_+ ,
\end{equation}
which we also write as $\cK_{S,+}$ by abuse of notation. We collect
below the following basic properties:
\begin{equation} \label{k-properties10}
\begin{aligned}
   \mbox{for}\ \xi \in \C \setminus
   \left(\left[-\frac{\omega}{\sqrt{\hm_0}},
     \frac{\omega}{\sqrt{\hm_0}}\right] \cup \ii \R\right) :&
   \qq q_S(\xi) = -q_S(-\xi) = -\overline{q_S(\overline{\xi})} ,
\\
   \mbox{for}\ \xi \in \left[-\frac{\omega}{\sqrt{\hm_0}},
     \frac{\omega}{\sqrt{\hm_0}}\right] :&
   \qq q_S(\xi \pm \ii 0)
   = \mp\left|\frac{\omega^2}{\hm_0}-\xi^2\right|^{1/2} ,
\\
   \mbox{for}\ \xi \in \ii\R :&
   \qq q_S(\xi \pm 0)
   = \mp\left|\frac{\omega^2}{\hm_0}-\xi^2\right|^{1/2} ,
\\
   \mbox{for}\ \xi \in \left(-\infty,
   -\frac{\omega}{\sqrt{\hm_0}}\right] \cup
      \left[\frac{\omega}{\sqrt{\hm_0}},+\infty\right) :&
   \qq q_S(\xi)
   = \pm \ii \left|\xi^2 - \frac{\omega^2}{\hm_0}\right|^{1/2} ,
   \qq \pm \xi \ge \frac{\omega}{\sqrt{\hm_0}} .
\end{aligned}
\end{equation}
Moreover, for $\xi$ in the cuts, we have
\begin{equation} \label{oncuts}
\begin{aligned}
   \mbox{for}\ \xi \in \left[0,\frac{\omega}{\sqrt{\hm_0}}\right] :&
   \qq q_S(\xi) = -\left|\frac{\omega^2}{\hm_0}-\xi^2\right|^{1/2} ,
\\
   \mbox{for}\ \xi \in \left[-\frac{\omega}{\sqrt{\hm_0}},0\right] :&
   \qq q_S(\xi) = \left|\frac{\omega^2}{\hm_0}-\xi^2\right|^{1/2} ,
\\
   \mbox{for}\ \xi \in \ii\R_\pm :&
   \qq q_S(\xi) = \mp\left|\frac{\omega^2}{\hm_0}-\xi^2\right|^{1/2} .
\end{aligned}
\end{equation}

By replacing $\mu_0$ with $\sigma_0 := \lambda_0 + 2 \hm_0$ we
get analogous properties for the quasimomentum
$$
   q_P(\xi) = \sqrt{\frac{\omega^2}{\sigma_0} - \xi^2}
   = \ii \sqrt{\xi^2 - \frac{\omega^2}{\sigma_0}} .
$$
$\cK_P$ and $\cK_{P,\pm}$ are defined in a manner similar to $\cK_S$
and $\cK_{S,\pm}$. We obtain the Riemann surface, $\cR$, for both
$q_P$ and $q_S$ by joining the Riemann surfaces for both quasimomenta
so that $q_P$ and $q_S$ are single-valued holomorphic functions of
$\xi$; $\cR$ is a four-fold cover of the complex plane. The sheets of
$\cR$,
$$
   {\cR} = {\cR}_{++} \cup {\cR}_{+-} \cup {\cR}_{-+} \cup {\cR}_{--}
   = \cup_{\sigma_1,\sigma_2} {\cR}_{\sigma_1,\sigma_2} ,\quad
   (\sigma_1,\sigma_2) = ({\rm sign} \Im q_P, {\rm sign} \Im q_S)
$$
are distinguished by the signs of the imaginary part of quasimomenta
$q_P$, $q_S$.

To a point $\xi \in \cR$ we may associate the two values $q_S(\xi)$,
$q_P(\xi)$ and can determine a mapping $\cR \to \cR$ by its action on
$q_S(\xi)$, $q_P(\xi)$. These mappings can be regarded as analogs of complex conjugation adapted to different sheets of the Riemann surface, and are used in order to define analytic continuation of certain identities initially defined on the branch cuts only. Thus, we define the mappings, $w_P$, $w_S$ and $w_{SP} :\ \cR \to \cR$
\begin{align}
   q_S(w_S(\xi)) = -q_S(\xi) ,&\qq q_P(w_S(\xi)) = q_P(\xi) ,
\label{S-jump}\\
   q_S(w_P(\xi)) = q_S(\xi) ,&\qq q_P(w_P(\xi)) = -q_P(\xi) ,
\label{P-jump}\\
   q_S(w_{SP}(\xi)) = -q_{S}(\xi) ,&\qq q_P(w_{SP}(\xi)) = -q_P(\xi). 
\label{SP-jump} 
\end{align}
These relations, between the sheets of the Riemann surface, map a
point $\xi \in\cR$ to another point in $\cR$ with the same projection,
$\Pi(\xi) \in \C$.

We identify $\cR_{++}$ where $\Im
q_P > 0$, $\Im q_S > 0$ with the physical (or ``upper'') sheet for
$q_S$ (cf.~(\ref{imq}-\ref{cKS})),
$$
   \cK_{S,\pm} = \left\{\xi \in \cK_S = \C \setminus
   \left(\left[-\frac{\omega}{\sqrt{\hm_0}},
     \frac{\omega}{\sqrt{\hm_0}}\right] \cup \ii \R\right)\ :\
   \Re \xi  > 0 \right\} .
$$
On $\cR_{++}$ we have $\Im q_P > \Im q_S$. Then
$$
   \Im q_P - \Im q_S = \Im (q_P - q_S)
   = \Im \frac{q_P^2 - q_S^2}{q_P + q_S}
   = -\omega^2 \frac{\lambda_0 + \mu_0}{(\lambda_0 + 2\mu_0) \mu_0}
     \Im \frac{1}{q_P + q_S} > 0 ,
$$
that is, $\Im (q_P + q_S)^{-1} < 0$. On the unphysical sheet, $\cR_{--}$, we have $\Im q_P < \Im q_S$. Then
$$
   \Im q_P - \Im q_S = -\omega^2
   \frac{\lambda_0 + \mu_0}{(\lambda_0 + 2 \mu_0) \mu_0}
   \Im \frac{1}{q_P + q_S} < 0 ,
$$
that is, $\Im (q_P + q_S)^{-1} > 0$. 

On the unphysical sheet, $\cR_{+-}$, where $\Im q_P > 0$ and $\Im q_S < 0$, we have $\Im q_P >
-\Im q_S$. Then
$$
   \Im q_P - \Im q_S = -\omega^2
   \frac{\lambda_0 + \mu_0}{(\lambda_0 + 2 \mu_0) \mu_0}
   \Im \frac{1}{q_P + q_S} > 0 ,
$$
that is, $\Im (q_P + q_S)^{-1} < 0$.

On the unphysical sheet, $\cR_{-+}$, where $\Im q_P < 0$ and $\Im q_S > 0$, we have $\Im q_S >
-\Im q_P$. Then
$$
   \Im q_P - \Im q_S = -\omega^2
   \frac{\lambda_0 + \mu_0}{(\lambda_0 + 2 \mu_0) \mu_0}
   \Im \frac{1}{q_P + q_S} < 0 ,
$$
that is, $\Im (q_P + q_S)^{-1} > 0$. We have the property,
$$
   \Re q_P \Im q_P = \Re q_S \Im q_S .
$$
For the later analysis, we introduce a function
\begin{multline} \label{eq:gammadef}
   \gamma(\xi) = \max\left\{
   \left[|\Im q_P(\xi)|-\Im q_P(\xi)\right],
   \frac12\left[|\Im (q_P(\xi) - q_S(\xi))|
     - \Im (q_P(\xi) - q_S(\xi))\right],\right.
\\
   \left.
   \frac12\left[|\Im (q_P(\xi) + q_S(\xi))|
     - \Im (q_P(\xi) + q_S(\xi))\right]\right\} .
\end{multline}
On $\cR_{++}$ we have $\Im q_P > \Im q_S > 0$ and, hence, $\gamma(\xi)
= 0$. On  $\cR_{--}$ we have $\Im q_P < \Im q_S < 0$, then
\begin{equation*}
   \gamma(\xi) = \max\left\{
   -2 \Im q_P(\xi),
   -\Im (q_P(\xi) - q_S(\xi)), \right.
   \left.
   -\Im (q_P(\xi) + q_S(\xi))\right\}
   \sim 2 |\xi|\quad\text{as}\ |\xi| \rightarrow\infty .
\end{equation*}
\noindent On $\cR_{+-}$ we have $\Im q_S < 0 < \Im q_P$ and, hence, $\gamma(\xi)
= 0$.  On $\cR_{-+}$ we have $\Im q_P < 0 < \Im q_S$, then
\begin{equation*}
   \gamma(\xi) = \max\left\{
   -2 \Im q_P(\xi),
   -\Im (q_P(\xi) - q_S(\xi)), \right.
   \left.
   -\Im (q_P(\xi) + q_S(\xi))\right\}
   \sim 2 |\xi|\quad\text{as}\ |\xi| \rightarrow \infty.
\end{equation*}

%From now on, we omit the subscript $S$ in the notation and write $\cK$
%for $\cK_{S}$. Likewise, we write $\cK_{+,+}$ for {\color{red} [is this correct? haven't we used such a notation already above in (3.18)? if so, please move this text]} $\cK_{S,+}$ for the cut plane corresponding to the physical sheet. We also introduce
%\[
%   \cK_{++,1} = \overline{\cK_+} \setminus
%              \left\{\frac{\omega}{\sqrt{\hm_0}}\right\} .
%\] 
%Passing through the cuts to the three unphysical sheets can be
%controlled using the maps $w_S$, $w_P$ and $w_{SP}$, where
%(\ref{S-jump}), (\ref{P-jump}) correspond to passing through exactly
%one branch cut and (\ref{SP-jump}) corresponds to passing through
%exactly two branch cuts.

\section{Conjugation properties of the Jost solutions and the
  boundary matrix}

Here, we analyze the symmetry properties of the Jost solutions on the cuts of the complex plane introduced in the previous section. For a complex function, $f$ say, on the projection of the Riemann surface $\cR$ to the cut plane $\cK$ we define the map $f \to f^*$ with $f^*(\xi) = \overline{f(\overline{\xi})}$. We obtain

\begin{lemma}[Conjugation of Jost solutions]\label{l-conjJ}
On the Riemann surface, $\cR$, the following holds true
\begin{align}
   &\psi_P^\pm(Z,w_{P}(\xi)) = \psi_P^\pm(Z,w_{PS}(\xi)) = \psi_P^\mp(Z,\xi) ,
\label{symmetry-1}\\
   &\psi_S^\pm(Z,w_{S}(\xi)) = \psi_S^\pm(Z,w_{PS}(\xi)) = \psi_S^\mp(Z,\xi) .
\label{symmetry-2}
\end{align}
On the projection of the Riemann surface $\cR$ to the cut plane $\cK$, we have
\begin{equation}
 (\Psi^{\pm})^*(Z,\xi)
   = \Psi^{\pm}(Z,\xi),\qq \xi\in\cK .
\label{symmetryproj}
\end{equation}
On the branch cut, $\xi\in\left[-\frac{\omega}{\sqrt{\sigma_0}},\frac{\omega}{\sqrt{\sigma_0}}\right]$,
\begin{equation}\lb{Jost_real_cut}
   \overline{\Psi^\pm(Z,\xi)}=\Psi^\mp(Z,\xi)
\end{equation}
and on the branch cut, $\xi \in \ii\R$,
\begin{equation}\lb{Jost_im_cut} \overline{\Psi^\pm(Z,\xi)}=\Psi^\pm(Z,\overline{\xi})=\Psi^\pm(Z,-\xi).
\end{equation}
\end{lemma}

\begin{proof}
For $\xi\in\cK$ we observe that equation (\ref{(1a)}) is invariant under operation $\psi \to \psi^*$. As $\overline{q_\bullet(\xi)}=-q_\bullet(\overline{\xi})$ and
for $Z<-H$ $$\overline{\psi_P^\pm(Z,\overline{\xi})}= \ma -\xi \\ \pm \ii q_P(\xi)\am e^{\pm \ii Z q_P}=\psi_P^\pm(Z,\xi),\qq  \overline{\psi_S^\pm(Z,\overline{\xi})}=\ma \pm \ii q_S  \\  \xi  \am e^{\pm \ii Z q_S}=\psi_S^\pm(Z,\xi).$$
Then (\ref{symmetryproj}) follows. 

For $\xi\in\left[-\frac{\omega}{\sqrt{\sigma_0}},\frac{\omega}{\sqrt{\sigma_0}}\right]$,
$$\overline{(\hat{H}_0(\xi) -\omega^2)\Psi}=( \hat{H}_0(\xi) -\omega^2)\overline{\Psi},$$
and for $Z\leq -H$ $$\overline{\psi_P^\pm(Z,\xi)}= \ma-\xi \\ \mp \ii q_P\am e^{\mp \ii Z q_P}=\psi_P^\mp(Z,\xi),\qq  \overline{\psi_S^\pm(Z,\xi)}=\ma \mp \ii q_S  \\ -  \xi  \am e^{\mp \ii Z q_S}=\psi_S^\mp(Z,\xi).$$
Then we get (\ref{Jost_real_cut}). 

For $\xi\in\ii\R$, also
$$\overline{( \hat{H}_0(\xi) -\omega^2)\Psi}=(\hat{H}_0(\xi) -\omega^2)\overline{\Psi},$$
and for $Z\leq -H$ $$\overline{\psi_P^\pm(Z,\xi)}= \ma \xi \\ \mp \ii q_P\am e^{\mp \ii Z q_P}=\psi_P^\pm(Z,-\xi),\qq  \overline{\psi_S^\pm(Z,\xi)}=\ma \mp \ii q_S  \\  \xi  \am e^{\mp \ii Z q_S}=\psi_S^\pm(Z,-\xi)$$ using (\ref{oncuts}). Then we get (\ref{Jost_im_cut}).
\end{proof}

Next, we reconsider the boundary matrix of tractions. This matrix satisfies the following identities:

\begin{lemma}[Conjugation of boundary matrix] 
On the projection of the Riemann surface $\cR$ to the cut plane $\cK$, we have
\begin{equation}\lb{Bconj}\mB^*(\xi)=\ma -\hat{a}(\psi_P^-)(\xi) &-\hat{a}(\psi_S^-)(\xi)\\ \hat{b}(\psi_P^-)(\xi) &\hat{b}(\psi_S^-)(\xi)\am,\qq \Delta^*(\xi) = -\Delta(\xi) .
\end{equation}
On the branch cut, $\xi\in\left[-\frac{\omega}{\sqrt{\sigma_0}},\frac{\omega}{\sqrt{\sigma_0}}\right]$,
\begin{equation}\lb{Bconj_real} \overline{\mB(\xi)} =\ma -\hat{a}(\psi_P^+)(\xi) &-\hat{a}(\psi_S^+)(\xi)\\ \hat{b}(\psi_P^+)(\xi) &\hat{b}(\psi_S^+)(\xi)\am
\end{equation}
and on the branch cut, $\xi \in \ii\R$,
\begin{equation}\lb{Bconj_im}\overline{\mB(\xi)}=\ma -\hat{a}(\psi_P^-)(\overline{\xi}) &-\hat{a}(\psi_S^-)(\overline{\xi})\\ \hat{b}(\psi_P^-)(\overline{\xi}) &\hat{b}(\psi_S^-)(\overline{\xi})\am=\ma -\hat{a}(\psi_P^-)(-\xi) &-\hat{a}(\psi_S^-)(-\xi)\\ \hat{b}(\psi_P^-)(-\xi) &\hat{b}(\psi_S^-)(-\xi)\am.
\end{equation}
\end{lemma}

\begin{proof}
Property (\ref{Bconj}) follows from (\ref{symmetryproj}) and from
\begin{align*}
&\hat{a}(\psi)^*(\xi)=-\ii\left(\lambda(0^-) \xi \psi_1(0^-,\xi) + (\lambda(0^-) + 2\mu(0^-)) \pdpd{\psi_2}{Z}(0^-,\xi)\right)=-\hat{a}(\psi)(\xi),
\\
&\hat{b}(\psi)^*(\xi)=-\xi \mu(0^-)\psi_2(0^-,\xi) + \mu(0^-)\pdpd{\psi_1}{Z}(0^-,\xi) = \hat{b}(\psi)(\xi)
\end{align*}
for $\xi\in\cK$.

For $\xi\in\left[-\frac{\omega}{\sqrt{\sigma_0}},\frac{\omega}{\sqrt{\sigma_0}}\right]$, using (\ref{Jost_real_cut}), we get
\begin{align*}
\overline{\hat{a}(\psi^\pm_P)}&=-\ii\left(\lambda(0^-) \xi \overline{\psi_{P;1}^\pm(0^-)} + (\lambda(0^-) + 2\mu(0^-)) \pdpd{\overline{\psi_{P;2}^\pm}}{Z}(0^-)\right)\\&=-\ii\left(\lambda(0^-) \xi \psi_{P;1}^\mp(0^-) + (\lambda(0^-) + 2\mu(0^-)) \pdpd{\psi_{P;2}^\mp}{Z}(0^-)\right) =-\hat{a}(\psi^\mp_P),\qq \overline{\hat{a}(\psi^\pm_S)}=-\hat{a}(\psi^\mp_S),
\\
\overline{\hat{b}(\psi^\pm_P)}&=-\xi \mu(0^-)\overline{\psi_{P;2}^\pm(0^-)} + \mu(0^-)\pdpd{\overline{\psi_{P;1}^\pm}}{Z}(0^-) = \hat{b}(\psi^\mp_P),\qq \overline{\hat{b}(\psi^\pm_S)}=\hat{b}(\psi^\mp_P).
\end{align*}
Then, for $\xi\in\left[-\frac{\omega}{\sqrt{\sigma_0}},\frac{\omega}{\sqrt{\sigma_0}}\right],$
\begin{align*}\overline{\mB(\xi)}=&\overline{\ma \hat{a}(\psi_P^-)(\xi) &\hat{a}(\psi_S^-)(\xi)\\ \hat{b}(\psi_P^-)(\xi) &\hat{b}(\psi_S^-)(\xi)\am} =\ma -\hat{a}(\psi_P^+)(\xi) &-\hat{a}(\psi_S^+)(\xi)\\ \hat{b}(\psi_P^+)(\xi) &\hat{b}(\psi_S^+)(\xi)\am\\=&\ma-1 & 0\\ 0 &1\am\ma \hat{a}(\psi_P^+)(\xi) &\hat{a}(\psi_S^+)(\xi)\\ \hat{b}(\psi_P^+)(\xi) &\hat{b}(\psi_S^+)(\xi)\am.\end{align*} and we get (\ref{Bconj_real}).

For $\xi\in\ii\R,$ using (\ref{Jost_im_cut}), we get 
\begin{align*}
\overline{\hat{a}(\psi^\pm_P)(\xi)}
&=-\ii\left(\lambda(0^-) (-\xi) \psi_{P;1}^\pm(0,-\xi) + (\lambda(0^-) + 2\mu(0^-)) \pdpd{\psi_{P;2}^\pm}{Z}(0^-,-\xi)\right)\\&=-\hat{a}(\psi^\pm_P)(-\xi),\qq \overline{\hat{a}(\psi^\pm_S)(\xi)}=-\hat{a}(\psi^\pm_S)(-\xi),
\\
\overline{\hat{b}(\psi^\pm_P)(\xi)}
&=-(-\xi) \mu(0^-) \psi_{P;2}^\pm(0^-,-\xi) + \mu(0^-)\pdpd{\psi_{P;1}^\pm}{Z}(0,-\xi)\\ &  =\phantom{-} \hat{b}(\psi^\pm_P)(-\xi),\qq \overline{\hat{b}(\psi^\pm_S)}(\xi)=\phantom{-}\hat{b}(\psi^\pm_P)(-\xi).
\end{align*}
Then, for $\xi\in\ii\R,$
$$\overline{\mB(\xi)}=\overline{\ma \hat{a}(\psi_P^-)(\xi) &\hat{a}(\psi_S^-)(\xi)\\ \hat{b}(\psi_P^-)(\xi) &\hat{b}(\psi_S^-)(\xi)\am} =\ma -\hat{a}(\psi_P^-)(-\xi) &-\hat{a}(\psi_S^-)(-\xi)\\ \hat{b}(\psi_P^-)(-\xi) &\hat{b}(\psi_S^-)(-\xi)\am$$
and we get (\ref{Bconj_im}).
\end{proof}

\subsection*{Decomposition into entire functions}
% \label{s-repr}

It is convenient to decompose the Jost functions, analytic on the Riemann surface, into the pairs of entire functions on $\C$ \cite[Section~5.2]{IantchenkoKorotyaev2013a}. These functions allow useful representations of the numerators of the reflection coefficients (\ref{Num}) and the Rayleigh determinant  (\ref{Rd}), satisfy algebraic property (\ref{l_algebra}), and will play an important role in the solution of the inverse problem. We postpone the proof of analytic properties of these functions until Section~\ref{ss-A}.

We introduce, with a slight abuse of notation for $\varphi$,
$$
   \vartheta_P = \frac12 \left(\psi_{P}^+ + \psi_{P}^-\right) ,\
   \varphi_P = \frac{1}{2q_P} \left(\psi_{P}^+ - \psi_{P}^-\right) ,\qq
   \vartheta_S = \frac12 \left(\psi_{S}^+ + \psi_{S}^-\right) ,\
   \varphi_S = \frac{1}{2q_S} \left(\psi_{S}^+ - \psi_{S}^-\right) .
$$

These functions are entire. Clearly,
\begin{equation} \label{eq:JtoEnt}
   \psi_{P}^\pm = \vartheta_{P} \pm q_P \varphi_{P} ,\qq
   \psi_{S}^\pm = \vartheta_{S} \pm q_S \varphi_{S} .
\end{equation}

\begin{remark}\lb{entirecond_0}
In case $H = 0$, we have a Rayleigh system with constant Lam{\'e} parameters, $\lambda_0$, $\mu_0$. Then, for $Z < 0$,
\begin{eqnarray*}
   & &\vartheta_{P,0}(Z,\xi) = \frac12\left(\psi_{P,0}^+ + \psi_{P,0}^-\right)(Z,\xi)
   = \left(\begin{array}{c} -\xi \cos(q_P Z) \\
     -\ii q_P \sin(Z q_P) \end{array}\right) ,
\\     
   & &\varphi_{P,0}(Z,\xi) = \frac{1}{2 q_P} \left(\psi_{P,0}^+ - \psi_{P,0}^-\right)(Z,\xi)
   = \left(\begin{array}{c} -\ii \frac{\xi}{q_P} \sin(q_P Z) \\
     \ii \cos(q_P Z) \end{array}\right) ,
\\
   & &\vartheta_{S,0}(Z,\xi) = \frac12\left(\psi_{S,0}^+ + \psi_{S,0}^-\right)(Z,\xi)
   = \left(\begin{array}{c} -q_S \sin(q_S Z) \\
     -\xi \cos(q_{S} Z) \end{array}\right) ,
\\
   & &\varphi_{S,0}(Z,\xi) = \frac{1}{2 q_S} \left(\psi_{S,0}^+ - \psi_{S,0}^-\right)(Z,\xi)
   = \left(\begin{array}{c} \ii \cos(q_S Z) \\
     -\ii \frac{\xi}{q_S} \sin(q_S Z) \end{array}\right) ,
\end{eqnarray*}
which yield the boundary conditions,
\begin{eqnarray*}
   \vartheta_{P,0}(0^-) = \ma -\xi \\ 0 \am ,\qq
   \pdpd{\vartheta_{P,0}}{Z}(0^-) = \ma 0 \\ -\ii q_P^2 \am &,\qq&
   \varphi_{P,0}(0^-) = \ma 0 \\ \ii \am ,\qq
   \pdpd{\varphi_{P,0}}{Z}(0^-) = \ma -\ii \xi \\0 \am ,
\\
   \vartheta_{S,0}(0^-) = \ma 0 \\ -\xi \am ,\qq
   \pdpd{\vartheta_{S,0}}{Z}(0^-) = \ma -q_S^2 \\ 0\am &,\qq&
   \varphi_{S,0}(0^-) = \ma \ii  \\ 0 \am ,\qq
   \pdpd{\varphi_{S,0}}{Z}(0^-) = \ma 0 \\ -\ii \xi \am .
\end{eqnarray*}
\end{remark}
\medskip
In general case, relating to the  homogeneous case in  Remark \ref{entirecond_0}, we get
\[\lb{entirecond}
   (\vartheta_P,\varphi_P,\vartheta_S,\varphi_S)
   = (\vartheta_{P,0},\varphi_{P,0},\vartheta_{S,0},\varphi_{S,0})\qq
   \text{for}\qq Z \le -H ,
\]
which can be considered as boundary conditions. So, we have
\begin{remark}\lb{r-entirecond}
Functions $\vartheta_P,\varphi_P,\vartheta_S,\varphi_S$ are (unique) solutions to the Rayleigh equation (\ref{(1a)}) satisfying the conditions
(\ref{entirecond}), which implies that they are entire on $\C.$
\end{remark}
\medskip

\noindent
We introduce the following notation relevant to the boundary matrix,
\begin{align}&\gamma_1:=\hat{a}(\vartheta_P)=\frac12\left(\hat{a}(\psi_{P}^+)+\hat{a}(\psi_{P}^-)\right),\qq\gamma_3:=\hat{a}(\varphi_{P})=\frac{1}{2q_P}\left(\hat{a}(\psi_{P}^+)-\hat{a}(\psi_{P}^-)\right),\label{eq:defgam13}\\
  &\gamma_5:=\hat{a}(\vartheta_{S})=\frac12\left(\hat{a}(\psi_{S}^+)+\hat{a}(\psi_{S}^-)\right),\qq\gamma_7:=\hat{a}(\varphi_{S})=\frac{1}{2q_S}\left(\hat{a}(\psi_{S}^+)-\hat{a}(\psi_{S}^-)\right)
\end{align}  
and
\begin{align}&\gamma_2:=\hat{b}(\vartheta_P)=\frac12\left(\hat{b}(\psi_{P}^+)+\hat{b}(\psi_{P}^-)\right),\qq\gamma_4:=\hat{b}(\varphi_{P})=\frac{1}{2q_P}\left(\hat{b}(\psi_{P}^+)-\hat{b}(\psi_{P}^-)\right),\\
  &\gamma_6:=\hat{b}(\vartheta_{S})=\frac12\left(\hat{b}(\psi_{S}^+)+\hat{b}(\psi_{S}^-)\right),\qq\gamma_8:=\hat{b}(\varphi_{S})=\frac{1}{2q_S}\left(\hat{b}(\psi_{S}^+)-\hat{b}(\psi_{S}^-)\right) .
\label{eq:defgam68}
\end{align} It follows immediately that $\gamma_1,\ldots,\gamma_8$ are entire as well.

\begin{remark}
The elements of the boundary matrix admit the decompositions into entire functions,
\begin{align*}&\hat{a}(\psi_{P}^+)=\gamma_1+q_P\gamma_3,\qq  \hat{a}(\psi_{P}^-)=\gamma_1-q_P\gamma_3,\qq \hat{a}(\psi_{S}^+)=\gamma_5+q_S\gamma_7,\qq \hat{a}(\psi_{S}^-)=\gamma_5-q_S\gamma_7\\
&\hat{b}(\psi_{P}^+)=\gamma_2+q_P\gamma_4,\qq  \hat{b}(\psi_{P}^-)=\gamma_2-q_P\gamma_4,\qq \hat{b}(\psi_{S}^+)=\gamma_6+q_S\gamma_8,\qq \hat{b}(\psi_{S}^-)=\gamma_6-q_S\gamma_8.
\end{align*}
\end{remark}

\medskip

\noindent
Furthermore, we introduce the $2 \times 2$ determinants
\begin{align*}d_1&:=\gamma_1\gamma_6-\gamma_5\gamma_2,\qq d_2:=-\gamma_3\gamma_6+\gamma_5\gamma_4,\\
d_3&:=-\gamma_1\gamma_8+\gamma_7\gamma_2,\qq 
d_4:=\gamma_3\gamma_8-\gamma_7\gamma_4 ,
\end{align*}
and
\begin{align*}
-\tfrac{1}{2} \cP&:= \gamma_3\gamma_2-\gamma_1\gamma_4,\qq
-\tfrac{1}{2} \cS:= \gamma_7\gamma_6-\gamma_5\gamma_8.
\end{align*}
We have the identity
\[\lb{Num}
   \cP \cS = 4 (d_1 d_4 - d_2 d_3) .
\]
We find that the Rayleigh determinant (cf.~(\ref{eq:Rdet})) can be expressed in the determinants introduced above,
\[\lb{Rd}
   \Delta = d_1 + q_P d_2 + q_S d_3 + q_P q_S d_4 .
\]

\begin{remark}
In the case of a constant half space, when $H = 0$, we have
\begin{align*}
  &d_1 = d_{1,0} =  \hat{a}(\vt_{P,0}) \hat{b}(\vt_{S,0})
  - \hat{a}(\vt_{S,0}) \hat{b}(\vt_{P,0})
  = \ii \hm_0^2 \left(\frac{\omega^2}{\hm_0} - 2 \xi^2\right) ,
  \\[-0.15cm]
  &d_2 = d_{2,0} = -\hat{a}(\vp_{P,0}) \hat{b}(\vt_{S,0})
  + \hat{a}(\vt_{S,0}) \hat{b}(\vp_{P,0}) = 0 ,
  \\
  &d_3 = d_{3,0} = -\hat{a}(\vt_{P,0}) \hat{b}(\vp_{S,0})
  + \hat{a}(\vp_{S,0}) \hat{b}(\vt_{P,0}) = 0 ,
  \\
  &d_4 = d_{4,0} = \hat{a}(\vp_{P,0}) \hat{b}(\vp_{S,0})
  - \hat{a}(\vp_{S,0}) \hat{b}(\vp_{P,0}) = \ii 4 \hm_0^2 \xi^2 ,
\end{align*}
while
\begin{equation*}
   \Delta(\xi) = \Delta_0(\xi) = \ii \hm_0^2 \left(\left(\frac{\omega^2}{
             \hm_0} - 2 \xi^2\right)^2 + 4 q_P q_S \xi^2\right) .
\end{equation*}
\end{remark}

\begin{lemma} \label{l_algebra}
The following algebraic relation holds true,
$$
\frac12\cS\ma \gamma_1\\ \gamma_2\\\gamma_3\\ \gamma_4\am=\left(\begin{array}{ccccccccc}  -d_3 & 0 & -d_1& 0\\
0 & -d_3 & 0 &-d_1\\
d_4 & 0&d_2& 0\\ 0 & d_4 & 0&d_2 \end{array}\right) \ma \gamma_5\\ \gamma_6\\ \gamma_7\\ \gamma_8  \am .
$$
\end{lemma}

The proof is straightforward.

 We now discuss some more properties of $\gamma_1,\ldots,\gamma_8$. We will always assume that $\omega>0$ is fixed.

\begin{lemma} \label{l-prop_delta}
On the branch cut, $\xi \in \left[-\frac{\omega}{\sqrt{\sigma_0}},\frac{\omega}{\sqrt{\sigma_0}}\right]$, we have
\begin{equation} \label{gamma_real_cut}
\gamma_1,\gamma_4,\gamma_5,\gamma_8\in\ii\R,\qq  \gamma_2,\gamma_3,\gamma_6,\gamma_7\in \R\qq \text{and}\qq
d_1,d_4\in\ii\R,\qq d_2,d_3,\cP,\cS\in\R.
\end{equation}
On the branch cut, $\xi \in \ii \R$, the following holds true,
\begin{align}&\overline{\gamma_1(\xi)}= -\gamma_1(-\xi),\qq \overline{\gamma_2(\xi)}= \gamma_2(-\xi),\qq \overline{\gamma_3(\xi)}= \gamma_3(-\xi),\qq \overline{\gamma_4(\xi)}= -\gamma_4(-\xi)\lb{gamma_im_cut}\\
  &\overline{\gamma_5(\xi)}= -\gamma_5(-\xi),\qq \overline{\gamma_6(\xi)}= \gamma_6(-\xi),\qq \overline{\gamma_7(\xi)}= \gamma_7(-\xi),\qq \overline{\gamma_8(\xi)}= -\gamma_8(-\xi)
\end{align}
and
\begin{align}
  \overline{d_1(\xi)}=-d_1(-\xi),\qq \overline{d_2(\xi)}=d_2(-\xi),\qq \overline{d_3(\xi)}=d_3(-\xi),\qq \overline{d_4(\xi)}=-d_4(-\xi) ;
\end{align}
furthermore,
\begin{align}
  \overline{\cP(\xi)}=\cP(-\xi),\qq
  \overline{\cS(\xi)}=\cS(-\xi)
\end{align}
and
\begin{align}
  \overline{\Delta(\xi)}=-\Delta(-\xi).
\end{align}
\end{lemma}

The proof is a straightforward consequence of Lemma \ref{l-conjJ}. 
For example,  from (\ref{Jost_real_cut}) it follows for $\xi\in\left[-\frac{\omega}{\sqrt{\sigma_0}},\frac{\omega}{\sqrt{\sigma_0}}\right]$ that 
\begin{align*}
\overline{\hat{a}(\psi^\pm_P)}=-\hat{a}(\psi^\mp_P),\qq \overline{\hat{a}(\psi^\pm_S)}=-\hat{a}(\psi^\mp_S),\qq
\overline{\hat{b}(\psi^\pm_P)} = \hat{b}(\psi^\mp_P),\qq \overline{\hat{b}(\psi^\pm_S)}=\hat{b}(\psi^\mp_P)
\end{align*} and then \begin{align*}&\overline{\gamma_1}= -\gamma_1,\qq \overline{\gamma_2}= \gamma_2,\qq \overline{\gamma_3}= \gamma_3,\qq \overline{\gamma_4}= -\gamma_4\\
&\overline{\gamma_5}= -\gamma_5,\qq \overline{\gamma_6}= \gamma_6,\qq \overline{\gamma_7}= \gamma_7,\qq \overline{\gamma_8}= -\gamma_8.
\end{align*}

Concluding this section, subjecting the Rayleigh determinant to conjugations, we find

\begin{lemma}
The Rayleigh determinant satisfies
$$
\Delta(\xi)\Delta(w_{PS}(\xi))-\Delta(w_P(\xi))\Delta(w_{S}(\xi))=q_P q_S\cP\cS,\qq\xi\in\cR .
$$
\end{lemma}

\begin{proof}
We have
\begin{multline*}
\Delta(\xi)\Delta(w_{PS}(\xi))=\left(d_1+q_Pq_S d_4\right)^2-\left(q_Pd_2+q_Sd_3\right)^2\\
=d_1^2-q_P^2d_2^2-q_S^2d_3^2+q_P^2q_S^2 d_4^2+2\left(d_1d_4-d_2d_3\right)q_Pq_S
\end{multline*}
and
\begin{multline*}
\Delta(w_P(\xi))\Delta(w_{S}(\xi))
= \left( d_1-q_Pd_2+q_Sd_3-q_Pq_Sd_4\right)\left( d_1+q_Pd_2-q_Sd_3-q_Pq_Sd_4\right)
\\
=d_1^2-q_P^2d_2^2-q_S^2d_3^2+q_P^2q_S^2 d_4^2-2\left(d_1d_4-d_2d_3\right)q_Pq_S .
\end{multline*}
Hence,
\begin{align*}
&\Delta(\xi)\Delta(w_{PS}(\xi))-\Delta(w_P(\xi))\Delta(w_{S}(\xi))=4\left(d_1d_4-d_2d_3\right)q_Pq_S
\end{align*} and as $d_1d_4-d_2d_3=\frac14\cP\cS$, while these equalities all hold for $\xi \in \cR$, we get the statement of the Lemma.
\end{proof}

\section{Reflection matrix and its properties}

\subsection{Reflected Jost solutions}

We send in from below the Jost solutions, $\psi^+_P$ or $\psi^+_S$,
generating the ``reflected'' solutions, $g^+_P$ or $g^+_S$ respectively, at $Z = 0$. We express the reflected solutions in terms of the
Jost solutions,
\begin{equation} \label{reflected_waves}
   g^+_P = \psi^+_P + R_2 \psi_P^- - q_P R_1 \psi_S^- ,\qq
   g^+_S = \psi^+_S + q_S \widetilde{R}_1 \psi^-_P + \widetilde{R}_2 \psi^-_S
\end{equation}
with the reflection coefficients $R_1$, $R_2$, $\widetilde{R}_1$ and $\widetilde{R}_2$ to be determined. We form the reflection matrix
\[\lb{R-m}
   \mR  = \ma R_2 & q_S \widetilde{R}_1 \\
             -q_P R_1 & \widetilde{R}_2 \am .
\]
We note that in \cite{CdV2006} the reflection matrix (in the homogeneous case) was defined as the transpose of $\mR$. Our choice is motivated by analogy of the scattering matrix on the whole line, connecting the incoming waves with the outgoing ones, as for example in \cite[2.4.2]{DyatlovZworski2022}.

The reflection coefficients are obtained by imposing the traction-free boundary conditions,
\[\lb{Cond1}
   \left\{\begin{array}{cc}
   R_2 \hat{a}(\psi^-_P) - q_P R_1 \hat{a}(\psi^-_S) =& -\hat{a}(\psi^+_P) ,\\
   R_2 \hat{b}(\psi^-_P) - q_P R_1 \hat{b}(\psi^-_S) =& -\hat{b}(\psi^+_P)
   \end{array}\right.
\]
and
\[\lb{Cond2}
   \left\{\begin{array}{cc}
   q_S \widetilde{R}_1 \hat{a}(\psi^-_P)
       + \widetilde{R}_2 \hat{a}(\psi^-_S) =& -\hat{a}(\psi^+_S) ,\\
   q_S \widetilde{R}_1 \hat{b}(\psi^-_P)
       + \widetilde{R}_2 \hat{b}(\psi^-_S) =& -\hat{b}(\psi^+_S) .
   \end{array}\right.
\]
We straightforwardly obtain
\begin{align}
   &R_2 = \frac{1}{\Delta}
   \det\ma -\hat{a}(\psi_P^+) & \hat{a}(\psi_S^-) \\ -\hat{b}(\psi_P^+) & \hat{b}(\psi_S^-)\am ,\qq
   R_1 = \frac{1}{q_P \Delta}
   \det\ma \hat{a}(\psi_P^-) & \hat{a}(\psi_P^+) \\ \hat{b}(\psi_P^-) & \hat{b}(\psi_P^+)\am ,
\label{detref}\\
   &\widetilde{R}_2 = \frac{1}{\Delta}
   \det\ma -\hat{a}(\psi_P^-) & \hat{a}(\psi_S^+) \\ -\hat{b}(\psi_P^-) & \hat{b}(\psi_S^+)\am ,\qq
   \widetilde{R}_1 = \frac{1}{q_S \Delta}
   \det\ma -\hat{a}(\psi_S^+) & \hat{a}(\psi_S^-) \\ -\hat{b}(\psi_S^+) & \hat{b}(\psi_S^-)\am .
\end{align} 
Revisiting the mappings $w_P$, $w_S$ and $w_{SP}$ once again, we find

\begin{lemma} \label{lem:RDel}
For $\xi \in \cR$, the following holds true,
\begin{equation*}
   R_1(\xi) = \frac{\cP(\xi)}{\Delta(\xi)} ,\qq
   \widetilde{R}_1(\xi) = \frac{\cS(\xi)}{\Delta(\xi)} ,\qq
   \widetilde{R}_2(\xi) = -\frac{\Delta(w_S(\xi))}{\Delta(\xi)} ,\qq
   R_2(\xi) = -\frac{\Delta(w_P(\xi))}{\Delta(\xi)} .
\end{equation*}
\end{lemma}

\begin{remark}
In the homogeneous half space case, when $H = 0$, 
\begin{multline*}
    \mR = \frac{1}{\Delta_0(\xi)}\ma -\Delta_0(w_P(\xi))&q_S\cS(\xi) \\  -q_P\cP(\xi)&-\Delta_0(w_S(\xi))\am
\\
    =\frac{1}{\Delta_0(\xi)}\ma -\Delta_0(w_P(\xi))&-\mu_0^2q_S4|\xi|\left(\frac{\omega^2}{\mu}-2 \xi^2\right)\\\mu_0^2q_P4|\xi|\left(\frac{\omega^2}{\mu}-2 \xi^2\right) &-\Delta_0(w_S(\xi))\am,
\end{multline*}
where
\begin{align*}
  \Delta_0(w_P(\xi)) = \Delta_0(w_S(\xi))=\ii \mu_0^2 \left(\left(\frac{\omega^2}{\mu{}_0}-2 \xi^2\right)^2-4 \xi^2q_Pq_S\right).
\end{align*}
We have that $\det\mR = 1$, which is not true in the general inhomogeneous case.
\end{remark}

We note that the elements of the reflection matrix can be expressed in
terms of the entire functions through (\ref{eq:JtoEnt}).

\subsection{ Representation of reflection matrix in terms of boundary matrix}

Equations (\ref{Cond1})-(\ref{Cond2}) can be conveniently written in matrix form,
$$\mB \mR=- \ma \hat{a}(\psi_P^+) &\hat{a}(\psi_S^+)\\ 
\hat{b}(\psi_P^+) &\hat{b}(\psi_S^+)\am.$$
We then obtain

\begin{lemma} We have the representation
\begin{equation}
   \mR(\xi) = -\mB^{-1}(\xi) \mB(w_{PS}(\xi)) ,\qq
\xi \in \cR.
\end{equation}
The determinants of the reflection matrix and the boundary matrix are related according to
$$
   \det\mR(\xi)
   = \frac{\Delta(w_{PS}(\xi))}{\Delta(\xi)}
   = (\Delta(\xi))^{-2}
   \left(\Delta(w_P(\xi)) \Delta(w_S(\xi)) + q_P(\xi) q_S(\xi) \cS(\xi) \cP(\xi) \right) .
$$
\end{lemma}

\medskip

\noindent
On the branch cut, $\xi\in\left[-\frac{\omega}{\sqrt{\sigma_0}},\frac{\omega}{\sqrt{\sigma_0}}\right]$, $\omega>0,$
$$
   \mR(\xi) = \mB^{-1}(\xi) \ma 1 & 0 \\ 0 & -1\am \overline{\mB(\xi)} ,
$$
which follows immediately from (\ref{Bconj_real}).

\subsection{Flux normalization and the Rayleigh determinant revisited}

We introduce the flux-normalized Jost solutions,
$$
   \widetilde{\psi}^{\pm}_P = -(\omega q_P)^{-1/2} \psi^{\pm}_P ,\qq
   \widetilde{\psi}^{\pm}_S =-\ii  (\omega q_S)^{-1/2} \psi^{\pm}_S
$$
and then
$$
   \tilde{g}_{P}^+=-\left(\omega q_P\right)^{-1/2}g_{P}^+,\qq \tilde{g}_{S}^+=-\ii \left(\omega q_S\right)^{-1/2}g_{S}^+ .
$$
The reflection matrix then takes the form
$$
 \widetilde{\mR} =
   \ma R_2 & \ii \sqrt{q_P q_S} \widetilde{R}_1 \\ \\
        \ii \sqrt{q_P q_S} R_1 & \widetilde{R}_2 \am .
$$
 The fundamental property of the Rayleigh boundary value problem is that $\widetilde{\mR}(\xi)$ is   unitary  for $\xi$ real and  $\Im q_S(\xi) = 0.$ So for $\xi\in\left[-\frac{\omega}{\sqrt{\sigma_0}}\frac{\omega}{\sqrt{\s_0}}\right]$,  $\omega>0,$
$$\frac{1}{|\Delta(\xi)|^2}\ma -\Delta(w_P(\xi))&\ii \sqrt{q_Pq_S}\cS(\xi) \\  \\ \ii \sqrt{q_Pq_S}\cP(\xi)&-\Delta(w_S(\xi))\am\ma \overline{-\Delta(w_P(\xi))}&\overline{\ii \sqrt{q_Pq_S}\cP(\xi)} \\  \\ \overline{\ii \sqrt{q_Pq_S}\cS(\xi)}&-\overline{\Delta(w_S(\xi))}\am=I_2 ,$$
which is equivalent to the following  identities:
\begin{equation} \label{RUid-1}
   |\sqrt{q_Pq_S}|^2|\cS(\xi)|^2+|\Delta(w_P(\xi))|^2=|\Delta(\xi)|^2 ,
\end{equation}
\begin{equation} \label{RUid-2}
   |\sqrt{q_Pq_S}|^2|\cP(\xi)|^2+|\Delta(w_S(\xi))|^2=|\Delta(\xi)|^2
\end{equation}
and
\begin{equation} \label{RUid-3}
   \ii \overline{\sqrt{q_Pq_S}}\Delta(w_P(\xi))\overline{\cP(\xi)}-\ii\sqrt{q_Pq_S}\overline{\Delta(w_S(\xi))}\cS(\xi)=0 .
\end{equation}
Here, by Lemma \ref{l-prop_delta},
\begin{align*}|\Delta|^2&=(d_1+q_Pd_2+q_Sd_3+q_Pq_Sd_4)(-d_1+q_Pd_2+q_Sd_3-q_Pq_Sd_4)\\
&=(q_Pd_2+q_Sd_3)^2-(d_1+q_Pq_Sd_4)^2,\\
|\Delta(w_P(\xi))|^2&= (d_1-q_Pd_2+q_Sd_3-q_Pq_Sd_4)(-d_1-q_Pd_2+q_Sd_3+q_Pq_Sd_4)\\
&=(q_Pd_2-q_Sd_3)^2-(d_1-q_Pq_Sd_4)^2,\\
|\Delta(w_S(\xi))|^2&= (d_1+q_Pd_2-q_Sd_3-q_Pq_Sd_4)(-d_1+q_Pd_2-q_Sd_3+q_Pq_Sd_4)\\
&=(q_Pd_2-q_Sd_3)^2-(d_1-q_Pq_Sd_4)^2=|\Delta(w_P(\xi))|^2.
\end{align*}
Using (\ref{RUid-1}) and (\ref{RUid-2}), we find that $\cS^2(\xi) = \cP^2(\xi)$. Moreover

\begin{lemma}  For $\xi\in\left[-\frac{\omega}{\sqrt{\s_0}},\frac{\omega}{\sqrt{\s_0}}\right],$  $\omega>0,$ the following holds true, $$\Delta(w_P(\xi))=- \overline{\Delta(w_S(\xi))},\qq \Delta(\xi)=- \overline{\Delta(w_{PS}(\xi))}.$$
\end{lemma}

\begin{proof}
As for  $\xi\in\left[-\frac{\omega}{\sqrt{\s_0}},\frac{\omega}{\sqrt{\s_0}}\right]$, we have that $d_1,d_4\in\ii\R$ and $d_2,d_3,q_P,q_S\in\R$,  identity $\Delta(w_P(\xi))=- \overline{\Delta(w_S(\xi))}$ follows from 
$$- \overline{\Delta(w_S(\xi))} =d_1-q_Pd_2+q_Sd_3-q_Pq_Sd_4=d_1-q_Pd_2+q_Sd_3-q_Pq_Sd_4.$$
Identity $\Delta(\xi)=- \overline{\Delta(w_{PS}(\xi))}$
follows from $$ d_1+q_Pd_2+q_Sd_3+q_Pq_Sd_4=\overline{-d_1+q_Pd_2+q_Sd_3-q_Pq_Sd_4}=-\overline{(d_1-q_Pd_2-q_Sd_3+q_Pq_Sd_4)}.$$
\end{proof}

\begin{corollary}We have $\cS(\xi)=-\cP(\xi)$ for $\xi\in\C,$  $\omega>0.$ 
\end{corollary}

\begin{proof}From $\overline{\sqrt{q_Pq_S}}\Delta(w_P(\xi))\overline{\cP(\xi)}-\sqrt{q_Pq_S}\overline{\Delta(w_S(\xi))}\cS(\xi)=0$ and $\cP,\cS \in\R$ for \\ $\xi\in\left[-\frac{\omega}{\sqrt{\s_0}},\frac{\omega}{\sqrt{\s_0}}\right]$, it follows that
$$\overline{\sqrt{q_Pq_S}}\cP(\xi)=-\sqrt{q_Pq_S}\cS(\xi)=0,$$
which by (\ref{oncuts}) using that $\xi\in\left[-\frac{\omega}{\sqrt{\s_0}},\frac{\omega}{\sqrt{\s_0}}\right]$, $q_S$ and $q_P$ have the same sign, implies  $\cS(\xi)=-\cP(\xi)$  on the branch cut $\left[-\frac{\omega}{\sqrt{\s_0}},\frac{\omega}{\sqrt{\s_0}}\right]$, and, hence, everywhere.
\end{proof}

Now, consider $ \xi\in\left(-\frac{\omega}{\sqrt{\mu_0}}, -\frac{\omega}{\sqrt{\sigma_0}}\right)\cup\left(\frac{\omega}{\sqrt{\sigma_0}}, \frac{\omega}{\sqrt{\mu_0}}\right)$,  $\omega>0,$ where $q_S\in\R,$  $q_P\in\ii \R.$ 
% \begin{align*}\overline{\mB(\xi)}=&\overline{\ma \hat{a}(\psi_P^-)(\xi) &\hat{a}(\psi_S^-)(\xi)\\ \hat{b}(\psi_P^-)(\xi) &\hat{b}(\psi_S^-)(\xi)\am} =\ma -\hat{a}(\psi_P^-)(\xi) &-\hat{a}(\psi_S^+)(\xi)\\ \hat{b}(\psi_P^-)(\xi) &\hat{b}(\psi_S^+)(\xi)\am\\=&\ma-1 & 0\\ 0 &1\am\ma \hat{a}(\psi_P^-)(\xi) &\hat{a}(\psi_S^+)(\xi)\\ \hat{b}(\psi_P^-)(\xi) &\hat{b}(\psi_S^+)(\xi)\am\end{align*} and, for $ \xi\in\left(-\frac{\omega}{\sqrt{\mu_0}}, -\frac{\omega}{\sqrt{\sigma_0}}\right)\cup\left(\frac{\omega}{\sqrt{\sigma_0}}, \frac{\omega}{\sqrt{\mu_0}}\right),$ 
Here, \\ $\overline{\Delta(\xi)}=-\Delta(w_S(\xi))$ and as 
$$\widetilde{R}_2=-\frac{\Delta(w_S(\xi))}{\Delta (\xi)},$$
we have $$|\widetilde{R}_2(\xi) |^2=\frac{\Delta(w_S(\xi))}{\Delta (\xi)}\frac{\Delta (\xi)}{\Delta(w_S(\xi))}=1,$$ showing that for  $ \xi\in\left(-\frac{\omega}{\sqrt{\mu_0}}, -\frac{\omega}{\sqrt{\sigma_0}}\right)\cup\left(\frac{\omega}{\sqrt{\sigma_0}}, \frac{\omega}{\sqrt{\mu_0}}\right),$ the amplitude of the $S$-to-$S$ reflection is one.

\section{Resolvent, its analytic continuation and poles} 

 We begin with expressing the Green's function or kernel of the resolvent in terms of the Jost solutions in

\begin{theorem}\label{Th_resolvent}
Let (cf. (\ref{reflected_waves})) 
$g_P^+=\psi_P^++R_2\psi_P^--q_PR_1\psi_S^-,\ g_S^+=\psi_S^++q_S\widetilde{R}_1\psi_P^-+\widetilde{R}_2\psi_S^-$. Then the Green's function for operator 
(\ref{(1a)}) subject to the traction-free boundary condition (with traction defined in (\ref{tract1}), (\ref{tract2})) is given by
\[\lb{resolvent}G(Z,Z',\xi)=\frac{1}{2 \ii \omega^2}\left\{\begin{array}{lc}\frac{1}{q_P}\psi_{P}^-(Z) (g_{P}^+(Z'))^{\rm T}+\frac{1}{q_S}\psi_{S}^-(Z)(g_{S}^+(Z'))^{\rm T},& Z< Z'<0,\\ \\\frac{1}{q_P}g_{P}^+(Z) ( \psi_{P}^-(Z'))^{\rm T}+\frac{1}{q_S}g_{S}^+(Z)(\psi_{S}^-(Z'))^{\rm T},& Z'< Z<0.\end{array} \right.\]  
\end{theorem}

\begin{remark}
The theorem shows that the form of the kernel of the resolvent in the nonhomogeneous case is the same as the one in the constant case (corresponding to  $H=0$, $\lambda(Z)=\lambda_0,$ $\mu(Z)=\mu_0$). That is, we recover the standard formula in this case \cite{Secher1998}. 
\end{remark}

\begin{proof} 
As announced in Section~\ref{sec:2}, we extend 
$\lambda$ and $\mu$ as even functions from $\R_-$ to $\R$ and introduce $\hat{\mathfrak{H}}_0(\xi)$; see also  Remark \ref{r_ext}. Thus we consider the (nonphysical) differential equation on the whole line,
\[\lb{(1a_ext)}
(  \hat{\mathfrak{H}}_0(\xi)-\omega^2)u:=(Pu')'+\xi (Nu'-(N^{\rm T}u)')+(\omega^2-\xi^2 M)u=0,\qq Z\in\R.
\]
Now we follow the construction of Stickler  
\cite{Stickler1986} valid for this equation for general (not necessarily even) Lam{\'e} parameters and $\xi \in\R$. The symmetry of the Lam{\'e} parameters allows us to evaluate the relevant Wronskian, denoted by $W_{12}$, explicitly (see Lemma \ref{l-W} below). As the (nonphysical) resolvent kernel on the whole line does not respect the traction condition at $Z=0$, we subtract a ``reflected'' kernel constructed using the reflected solutions in (\ref{reflected_waves}) to obtain the physical kernel of the physical resolvent.

As in \cite{Stickler1986} we consider the Jost solutions, $U_1$, $U_2$, $U_3$, $U_4$, of (\ref{(1a_ext)}) satisfying the following conditions outside the interval $[-H,H]$:
\begin{align}&
U_1 = U_{1,0} = a_1(\xi) e^{\ii \sigma(\xi) Z},\qq 
%\Psi^{r,-}=
U_4 = U_{4,0} = a_4(\xi) e^{-\ii \sigma(\xi) Z}\qq
\text{for}\,\, Z \geq H,\label{>H}\\
&
%\Psi^{l,-}=
U_2 = U_{2,0} = a_2(\xi) e^{-\ii \sigma(\xi) Z},\qq 
%\Psi^{l,+}=
U_3 = U_{3,0} = a_3(\xi) e^{\ii \sigma(\xi) Z}\qq\text{for}\,\, Z\leq -H,\label{<-H}
\end{align}
where
$$
a_1=\ma -\xi &\ii q_S\\ \ii q_P & -\xi \am,\qq a_4=\ma -\xi &-\ii q_S\\ -\ii q_P & -\xi \am,\qq a_2=a_4,\qq a_3=a_1,\qq \sigma=\ma q_P &0 \\ 0 & q_S\am,
$$
where $U_2$ corresponds to $\Psi$ for $Z \le 0$. We now use the invariance of the differential equation and find that
$$
   U_{1,0}(-Z,-\xi) = -U_{1,0}(Z,\xi) ,\qq 
   U_{4,0}(-Z,-\xi) = -U_{4,0}(Z,\xi)
$$
so that
\begin{align}
   U_1(Z,\xi) = -U_3(-Z,-\xi) ,\qq
   U_4(Z,\xi) = -U_2(-Z,-\xi) .
\end{align}
Moreover, using that $a_m(-\xi) = -a_m(\xi)$, $m = 1,\dots,4$ (recalling that $q_S(-\xi) = -q_S(\xi)$ as was stated in the first line of
(\ref{k-properties10}) and also in (\ref{oncuts})), and that $\sigma(-\xi) = -\sigma(\xi)$ we get 

\begin{lemma}
The Jost solutions $U_1$, $U_2$, $U_3$ and $U_4$ of (\ref{(1a_ext)}) satisfy
\[\lb{symmerty}
   U_1(Z,\xi) = U_{1,0}(Z,\xi)=a_1(\xi) e^{\ii \sigma(\xi) Z} ,\ U_4(Z,\xi) = U_{4,0}(Z,\xi) = a_4(\xi) e^{-\ii \sigma(\xi) Z}\ \text{for}\,\, Z \ge H .
\]
\end{lemma}

We define the Wronskian
$$W_{n,m}=U_m^{\rm T} P U_\lambda'- \left(U_m^{\rm T}\right)' P U_\lambda+\xi U_m^{\rm T} \left(N-N^{\rm T}\right) U_n,\qq m, n = 1,2,3,4.$$
Then $$W_{n,m}' = 0 . $$
Moreover, 
\begin{align*}
W_{n,n}=0,\qq W^{\rm T}_{n,m}=-W_{m,n} .
\end{align*}
By (\ref{symmerty}), $$ U_1(Z,\xi)=U_{1,0}(Z,\xi)=a_1(\xi) e^{\ii \sigma(\xi) Z}\qq\mbox{for}\,\, Z\leq -H.$$
Then, as the Wronskian $W_{12}$ is constant, we can compute it for $Z\leq -H$ where $U_1,U_2$ are explicitly  known and we get

\begin{lemma}\label{l-W}
The $1,2$-component of the Wronskian is given by $$W_{1,2}=2\ii\omega^2 \ma q_P & 0\\ 0 & q_S\am.
$$
\end{lemma}

\noindent
Now, Stickler \cite{Stickler1986} found the following formula  for the Green's function on the whole line 
\begin{equation}\lb{(24)}G_i(Z,Z',\xi)=\left\{\begin{array}{lr}
U_1(Z,\xi) W_{1,2}^{-1}(\xi)U_2^{\rm T}(Z',\xi), & Z> Z',\\
U_2(Z,\xi) (W_{1,2}^{-1} (\xi))^{\rm T} U_1^{\rm T}(Z',\xi), & Z< Z'.
\end{array}\right.
\end{equation}

We return to the vanishing traction boundary condition at $Z=0$, and impose it by introducing a contribution $G_r$ to the resolvent kernel.
We write $U_2=[u_{2,P},u_{2,S}]=[\psi_P^-\,\,\psi_S^-]$, $U_3=[u_{3,P},u_{3,S}]=[\psi_P^+\,\,\psi_S^+].$  The reflected waves (\ref{reflected_waves} )  at $Z=0$ (satisfying the Neumann condition) are then given by
$$g_P^+=u_{3,P}+R_2u_{2,P}-q_PR_1u_{2,S}, \qq g_S^+=u_{3,S}+q_S\widetilde{R}_1u_{2,P}+\widetilde{R}_2u_{2,S} .$$
We put $G_1:=[g_P^+\,\,g_S^+]$ and $U_2=\Psi^{-}.$  Then we may follow the standard refection method \cite[Section~2, p.706]{Secher1998} and the Appendix in \cite{DermenjianGaitan2000} in the case of constant Lam{\'e} parameters. On the half-line $\R^-$ we construct the resolvent (respecting the Neumann condition) in the form
$G=G_i+G_r,$ where $G_r$ satisfies
$(\hat{H}_0(\xi)-\omega^2) G_r=0$ and
$$ \hat{a}(G_r)=-\hat{a}(G_i),\qq \hat{b}(G_r)=-\hat{b}(G_i).$$
Carrying out the necessary calculations, we find that
\begin{equation}\label{24BC}G(Z,Z',\xi)=\left\{\begin{array}{lr}
G_1(Z,\xi) W_{12}^{-1}(\xi)U_2^{\rm T}(Z',\xi), & Z'<Z<0,\\
U_2(Z,\xi) (W_{12}^{-1} (\xi))^{\rm T}G_1^{\rm T}(Z',\xi), & Z< Z'<0.
\end{array}\right.
\end{equation}
Now, using the explicit form of the Wronskian $W_{1,2}$ given in Lemma \ref{l-W} we can write (\ref{24BC}) as in the statement of the theorem.
\end{proof} 

\medskip

\noindent
The explicit formula for the kernel of the resolvent, $G(Z,Z',\xi)$, and the analytic properties of the Jost solutions and by implication the boundary matrix presented in the next section, show that the resolvent has an analytic continuation into the Riemann surface $\cR$. Moreover, (apart of the branch points for $q_P$, $q_S$) the poles of the resolvent are exactly the roots of the Rayleigh determinant, $\Delta$, that is, the wavenumber resonances.

\section{Analytic properties of the boundary matrix}

Here, we finally study the analytic properties of the boundary matrix $\mB(\xi)$. To this end, we transform the Rayleigh system to a Schr{\"o}dinger form using the  Markushevich substitution (see Appendix~\ref{sec:Mark} and previous work \cite{MdHoopIantchenko2022}). This substitution relates the Jost solutions $\Psi^\pm(Z,\xi)$ to the Jost solutions $\bF^\pm(x,\xi)$ of the Schr{\"o}dinger type problem \cite{MdHoopIantchenko2022} with $x = -Z \in \R_+$, 
\[\lb{MarkTr}\Psi^\pm=[\psi_P^\pm\,\,\psi_S^\pm]=-\left(\xi\frac{\mu_0}{\omega^2}\right)^{-1}\ma 1  & 0\\ 0 & -1\am \gM^{-1}  \left(\bF^\mp\right) \ma 1 & 0\\ 0 & -1\am,\]
where first-order matrix-valued differentional operator $\gM^{-1}$ is defined in (\ref{1rw(2.10)new}). The Schr{\"{o}}dinger form is given by
\[\lb{MarkSchr}
\left\{\begin{array}{rl}-F''+VF+Q_0F&=-\xi^2F,\qq x\in\R_+\\
F'+\Theta(\xi) F&=0,\qq x=0 .
\end{array}\right.
\]
Here, $$V(x) = Q(x) - Q_0(x)$$ stands for the perturbation potential which satisfies $V(x) = 0$ for $x \geq H$. The potential $Q$ and the reference (background) potential $Q_0$ are defined in (\ref{1rw(2.8)new}) and (\ref{Q0}), respectively; $\Theta$ is defined in (\ref{1rw(2.15)new}).
% In the remainder of this section, we will write $\bF$ for $\bF^+ = \left[F^+_P\,\, F^+_S\right]$.

The boundary matrix $\mB(\xi)$ is explicitly related to the Jost function, given by  $$\bF_\Theta (\xi):=\bF'(0,\xi)+\Theta(\xi) \bF(0,\xi),\qq \bF = \bF^+;$$ that is,
\begin{align} \label{BFTheta}
 \bF_\Theta(\xi) =\frac{1}{ 2\mu(0)\omega^2}\ma -\mu(0)& 0 \\ \\ -2\mu_0\frac{\mu'(0)}{\mu(0)} &2\mu_0\xi  \am \mB(\xi) \ma \ii & 0 \\ \\ 0 & -1 \am.
\end{align}
Hence, analytic properties of the Jost function $\bF_\Theta(\xi)$, which we derive below using the Schr{\"o}dinger form of equation in (\ref{MarkSchr}), are directly translated to those of $\mB(\xi)$.  

\subsection{Analytic properties of the Jost solutions on the Riemann surface}

The matrix Jost solutions of (\ref{MarkSchr}) are determined by (radiation) boundary conditions
$$\bF^+(x,\xi)=\left[F^+_P(x,\xi)\,\, F^+_S(x,\xi)\right]=\bF^+_0(x,\xi)=[F^+_{P,0}(x,\xi)\,\, F^+_{S,0}(x,\xi)],\qq x\geq H,$$ 
where $F^+_{P,0},$ $F^+_{S,0}$ are solutions to (\ref{MarkSchr}) for $x>H$ of the form 
\begin{align*}&F^\pm_{P,0}(x,\xi)= \left(\begin{array}{c}
\left(-\frac{c_0}{2}G_{11}^H(x-H)+G_{21}^H\right) \pm \ii q_P(\xi)\frac{\mu_0}{\omega^2}G_{11}^H 
\\ \\ 
\left(-\frac{c_0}{2}G_{12}^H(x-H)+G_{22}^H\right) \pm \ii q_P(\xi)\frac{\mu_0}{\omega^2}G_{12}^H
\end{array}\right)e^{\pm \ii x q_P(\xi)},\qq c_0 = \frac{\lambda_0 + \mu_0}{\lambda_0 + 2 \mu_0},\\
& F^\pm_{S,0}(x,\xi)=-\frac{\mu_0\xi}{\omega^2}
\left(\begin{array}{c}
G_{11}^H
 \\  \\
G_{12}^H
\end{array}\right)e^{\pm \ii x q_S(\xi)},
\end{align*}
in which $G^H$ is given in Appendix~\ref{sec:Mark}. They are defined on the cut complex plane, $\cK$.

The Jost solutions satisfy the Volterra type integral equation, writing all the arguments noting that $V$ does not depend on $\xi$
\[\lb{Volterra}\bF^\pm(x,\xi)=\bF_{0}^\pm(x,\xi)-\int_x^\infty \bG(x,y;\xi) V(y)\bF^\pm(y,\xi)dy,\]
where $\bG(x,y;\xi)$ is the Green's function; each column of $\bG(.,y;\xi)$  satisfies \[\lb{AStLunpert}-\bF''+Q_0\bF=-\xi^2\bF\]
and the conditions $$\bG(x,x;\xi)={\mathbf 0},\qq\frac{\partial}{\partial x}\bG(x,y;\xi)|_{y=x}={\mathbf I}_2.$$
 The expression for
$Q_0$ is
\begin{align*}
  Q_0(x) =& \omega^2 \ma \displaystyle - \frac{1}{\mu_0} & 0 \\
  0 & \displaystyle -\frac{1}{\lambda_0 + 2 \mu_0} \am
\nonumber\\[0.25cm] &\qquad
  + \omega^2\frac{c_0}{\mu_0}
  \ma \displaystyle
  -G_{12}^H \left[-\frac{c_0}{2}G_{11}^H(x - H) + G_{21}^H\right] &
  \displaystyle 
  G_{11}^H \left[-\frac{c_0}{2} G_{11}^H(x - H) + G_{21}^H\right]
  \\[0.25cm]
  \displaystyle
  -G_{12}^H \left[-\frac{c_0}{2} G_{12}^H(x - H) + G_{22}^H\right] &
  \displaystyle 
  G_{12}^H \left[-\frac{c_0}{2} G_{11}^H(x - H) + G_{21}^H\right]
  \am.
\end{align*}
The Green's function is entire in $\xi\in\C$ and has the form given in

\begin{lemma}\label{l-ABC}
We have
\begin{multline}\bG(x,y;\xi)
= \bA(x)\frac{\sin((x-y)q_P(\xi))}{q_P(\xi)}
\\
+\bB(y)\frac{\sin((x-y)q_S(\xi))}{q_S(\xi)}+\bC\frac{\cos( (x-y)q_S(\xi))-\cos((x-y) q_P(\xi))}{\omega^2},
\end{multline}
where
\begin{align}
\bA(x) =& \left(\begin{array}{cc} G_{12}^H\left[\frac{c_0}{2}G_{11}^H(x-H)-G_{21}^H\right] &
G_{11}^H\left[-\frac{c_0}{2}G_{11}^H(x-H)+G_{21}^H\right]
\\ \\
G_{12}^H\left[\frac{c_0}{2}G_{12}^H(x-H)-G_{22}^H \right] &G_{11}^H\left[-\frac{c_0}{2}G_{12}^H(x-H)+G_{22}^H\right]
\end{array}\right),
\\
\bB(y) =& \left(\begin{array}{cc}
 G_{11}^H\left[\frac{c_0}{2}(-y+H)G_{12}^H+G_{22}^H\right] & -G_{11}^H\left[\frac{c_0}{2}(-y+H)G_{11}^H+G_{21}^H\right]
 \\ \\
 G_{12}^H\left[\frac{c_0}{2}(-y+H)G_{12}^H+G_{22}^H\right] & -G_{12}^H\left[\frac{c_0}{2}(-y+H)G_{11}^H+G_{21}^H\right]
\end{array}\right),
\\
\bC =& \left(\begin{array}{cc}
\mu_0G_{12}^HG_{11}^H &-\mu_0(G_{11}^H)^2 
\\ \\
\mu_0(G_{12}^H)^2 &-\mu_0G_{12}^HG_{11}^H 
\end{array}\right)
\end{align}
with the property that
\[\label{ABC}
\bA(x)+\bB(y) +\bC\cdot(y-x)\frac{ c_0}{2\mu_0}=\mathbf{I}_2
.\]
\end{lemma}

\noindent
Upon factoring out the complex exponentials in the Jost solutions, we obtain the Faddeev solutions
\begin{align*}
&H_P^\pm= e^{\mp \ii x q_P} F_P^\pm,\qq
 H_{P,0}^\pm= e^{\mp \ii x q_P} F_{P,0}^\pm,
\\
&H_S^\pm= e^{\mp \ii x q_P} F_S^\pm,\qq
 H_{S,0}^\pm= e^{\mp \ii x q_P} F_{S,0}^\pm.
\end{align*}
We find that
\begin{equation}\lb{Faddeev0}
  H_{P,0}^+ = 
\left(\begin{array}{c}
G_{21}^H \pm \ii q_P\frac{\mu_0}{\omega^2}G_{11}^H 
\\ \\
G_{22}^H \pm \ii q_P\frac{\mu_0}{\omega^2}G_{12}^H\end{array}\right)\qq\text{and}\qq
 H_{S,0}^+ =-\frac{\mu_0 \xi}{\omega^2}
e^{\ii (q_S-q_P)x}\left(\begin{array}{c}
G_{11}^H 
\\ \\
G_{12}^H
\end{array}\right).
\end{equation}

With the Green's function,
\begin{equation}
   \widetilde{\mathbf G}(x,y;\xi)=e^{\ii (y-x) q_P(\xi)} \bG(x,y;\xi),
\end{equation}
we obtain the Volterra equation 
\begin{equation}
   \bH^+(x,\xi)=\bH_0^+(x,\xi)-\int_x^\infty \widetilde{\mathbf G}(x,y;\xi) V(y) \bH^+(y;\xi)dy
\lb{V_P}
\end{equation}
for $\bH^+:=[H_P^+\,\, H_S^+ ]$, replacing (\ref{Volterra}). Employing classical techniques (see, for example \cite[Section~2.4]{BerezinShubin1991}), the growth of $\widetilde{\mathbf G}(x,y;\xi)$ in $\xi \in \cR$ can be estimated:

\begin{lemma}\lb{l-kernel}
Let $\omega>0$ be fixed. Then, for $x < y < H$ (within the slab)
\begin{align*}&\|\widetilde{\bG}(x,y;\xi)\|<\frac{C}{\max\{|\xi|,1\}} e^{(H-x)_+\gamma(\xi)},\end{align*}
where the constant $C$ only depends on the coefficients of the matrices $\bA$, $\bB$ and $\bC$ and $\|.\|$ is max matrix norm; $\gamma$ was given in (\ref{eq:gammadef}).
\end{lemma} 

Now, we can solve the Volterra equation (\ref{V_P}) by iteration and get a successive approximation of the solution. We follow essentially the procedure described in \cite[Section 2.4]{BerezinShubin1991} on the physical sheet and in \cite{Korotyaev2004, IantchenkoKorotyaev2014a, IantchenkoKorotyaev2014b} on the unphysical sheets. By estimating the terms using Lemma \ref{l-kernel} and changing the order of integrations in multiple integrals we gain a factorial factor in the denominator which makes the series converge uniformly on bounded sets excluding the branch points. On the unphysical sheets special care is needed to address the increasing exponential factors. This is where the boundedness of the zero order iteration, $\bH_0^+$, with columns given in (\ref{Faddeev0}) is used, together with Lemma \ref{l-kernel}.

Returning to the Jost solutions and  using the Riemann-Lebesgue lemma, we get

\begin{theorem}\label{th-Jost-solutions} Let  $\omega>0$ be fixed. Let $\xi\in\cR$ or its projection $\xi\in\Pi(\cR)\sim \cK$ and $\gamma$ be defined as in (\ref{eq:gammadef}).
For any fixed $x \geq 0$, the Jost solution, $\bF^+(x,\xi)$, is analytic on $\cR$ and $\cK$, of exponential type, satisfying
\begin{equation}
\bF^+(x,\xi) = \bF_0^+(x,\xi) - \int_x^H \bG(x,y) V(y)\bF_0^+(y,\xi)dy+\sum_{k=2}^\infty \bF_k(x,\xi),
\end{equation}
where
\begin{equation}
\bF_k(x,\xi)=\frac{|\xi|}{k!}\mathcal{O}\left(\frac{1}{\max\{|\xi|,1\}} \right)^ke^{(H-x)_+\gamma(\xi)}e^{-\Im q_S(\xi) x}.\end{equation}
In the limit  $|\xi| \rightarrow \infty$, $\xi \in \cK_{++}$, 
\begin{align}
\bF^+(x,\xi)=&-e^{-x\xi} \xi \frac{\mu_0}{\omega^2}\ma G_{11}^H&G_{11}^H\\ \\
G_{12}^H&G_{12}^H\am
\lb{xilimitJost}\\
&\left.+
e^{-x\xi} \frac{\mu_0}{\omega^2} \left\{\ma \frac12 G_{11}^H\left(c_0 H-x\right)+G_{21}^H& -\frac12x G_{11}^H\\ \\ \frac12 G_{12}^H\left(c_0 H-x\right)+G_{22}^H&  -\frac12xG_{12}^H  \am-\frac{1}{2}\frac{\mu_0}{\omega^2}\int_x^HV(y) \ma G_{11}^H &G_{11}^H\\ \\
G_{12}^H &G_{12}^H\am dy\right\}\right.\nonumber\\[0.25cm]
&+ o\left(1\right)
e^{-x \xi}
\nonumber
\end{align}
and
\begin{align}(\bF^+)'(x,\xi)=&e^{-x\xi}\xi^2\frac{\mu_0}{\omega^2}\ma G_{11}^H&G_{11}^H\\ \\
G_{12}^H&G_{12}^H\am\lb{xilimitDJostmoredoublebis}\\&-e^{-x\xi}\xi\left\{\ma \frac12 G_{11}^H\left(c_0 H-x\right)+G_{21}^H& -\frac12x G_{11}^H\\ \\ \frac12 G_{12}^H\left(c_0 H-x\right)+G_{22}^H&  -\frac12xG_{12}^H  \am-\frac{1}{2}\frac{\mu_0}{\omega^2}\int_x^H V(y)  \ma G_{11}^H &G_{11}^H\\ \\
G_{12}^H &G_{12}^H\am dy\right\}\nonumber\\[0.25cm]
&+ o\left(|\xi| \right)e^{-x\xi}.\nonumber\end{align}
\end{theorem}

By lengthy calculations, we get the following asymptotic estimate for the determinant of the relevant Jost solution at $x=0$. We write
$$\ma v_1 & v_2\\ \\ v_3 & v_4\am=\frac{1}{2}\int_0^HV(y)dy\frac{\mu_0 }{\omega^2}\ma G_{11}^H & G_{11}^H
\\ \\
G_{12}^H&G_{12}^H\am,\qq \mbox{satisfying}\,\,v_1-v_2=v_3-v_4=0,$$
and
\begin{align*}\ma w_1 & w_2\\ \\ w_3 & w_4\am=\frac{1}{2\xi}\int_0^HV(y) \ma -\frac{c_0}{2} G_{11}^H\left(y-H\right)+G_{21}^H& -\frac12y c_0 G_{11}^H\\ \\ -\frac{c_0}{2} G_{12}^H\left(y-H\right)+G_{22}^H&  -\frac12y c_0G_{12}^H   \am dy.
\end{align*}
Then
\begin{align*}\det\bF^+(0,\xi)=\xi\frac{\mu_0}{\omega^2} &+\frac{\mu_0 }{\omega^2}G_{11}^H\left(w_3-w_4\right)+ \frac{\mu_0 }{\omega^2}G_{12}^H\left(w_2-w_1\right)
\\
-&\left( \frac12 G_{11}^Hc_0 H+G_{21}^H\right)v_4+ \left( \frac12 G_{12}^Hc_0 H+G_{22}^H\right)v_2+o(1).\end{align*}

\subsection{Analytic properties of the Jost function}

We define the Weyl matrix, $\bM$ as \cite{MdHoopIantchenko2022} 
\begin{equation*}
   \bM(\xi) = \bF^+(0,\xi) [\bF_\Theta (\xi)]^{-1} .
\end{equation*}
Thus
$$\det\bF_\Theta(\xi)=\det\bF^+(0,\xi)\det \bM^{-1}(\xi),\qq \xi\in\cK_{+,+}.$$
Expanding the determinant of (108) in previous work  \cite[Lemma~VI.1]{MdHoopIantchenko2022}, we find that
\begin{align*}
  \det \bM^{-1}(\xi)=&\xi^2\left(1- 2\varpi\theta_2\right)+\xi \left(\theta_3+\varpi Q_{12}\right) +{\mathcal O}\left(1\right),\\ &\varpi=\frac{\mu_0}{\mu(0)},\qq 1 - 2 \varpi
  \theta_2 = c(0) = \frac{\lambda(0) + \mu(0)}{\lambda(0) + 2
    \mu(0)} ,\qq \theta_3=\frac{\dot{\mu}(0)}{\mu(0)},
\end{align*}  
where we note the appearance of a potential matrix element. We get the following asymptotic expansions for the determinant of the Jost function and the Rayleigh determinant in

\begin{lemma}[Physical sheet $\cK_{+,+}$] Let  $\omega>0$ be fixed. We have, for $\Im q_P\geq 0,$ $\Im q_S\geq 0,$ as $|\xi|\rightarrow\infty$ (so also as $\xi\in \ii \R,$ $|\xi|\rightarrow\infty$),
\begin{multline*}
\det\bF_\Theta(\xi)=\xi^3 c(0)\frac{\mu_0}{\omega^2}+\xi^2\left(\frac{\mu_0}{\omega^2}\left(\theta_3+\varpi Q_{12}\right)\right.\\
\left.+c(0)\left(\frac{\mu_0 }{\omega^2}G_{11}^H\left(w_3-w_4\right)+ \frac{\mu_0 }{\omega^2}G_{12}^H\left(w_2-w_1\right)- \left(\frac12 G_{11}^H c_0 H + G_{21}^H\right) v_4+ \left( \frac12 G_{12}^H c_0 H + G_{22}^H\right) v_2 \right)\right)
\\
+ o(\xi^2)
\end{multline*}
and
\begin{align*}
   \Delta = \det\mB = \ii \xi^2 A - \ii \xi B + o(\xi),
\end{align*}
where
\begin{equation}
   A = -2\omega^2 \mu(0) c(0)
\end{equation}
and
\begin{multline}
   B=\frac{\omega^42\mu(0)}{ \mu_0}\left(\frac{\mu_0}{\omega^2}\left(\theta_3+\varpi Q_{12}\right)\right.
\\
   \left.+ c(0)\left(\frac{\mu_0 }{\omega^2}G_{11}^H\left(w_3-w_4\right)+ \frac{\mu_0 }{\omega^2}G_{12}^H\left(w_2-w_1\right)-\left(\frac12 G_{11}^Hc_0 H+G_{21}^H\right)v_4 + \left( \frac12 G_{12}^Hc_0 H+G_{22}^H\right)v_2\right)\right).
\end{multline}
\end{lemma}

Next, we consider the unphysical sheets. We note that the Jost solution $\bF(x,\xi)$ is bounded on the physical sheet $ \cK_{++}$ and unbounded on the unphysical sheets. We denote with subindex $\sigma_1,\sigma_2$ the continuation of Jost solution and function to the Riemann sheets $\cK_{\sigma_1,\sigma_2}$, $\sigma_1,\sigma_2 \in \{\pm 1\}$. We have \[\lb{identification1}\bF_{\Theta,+-}(\xi)=\bF_\Theta(w_S(\xi)),\qq \bF_{\Theta,-+}(\xi)=\bF_\Theta(w_P(\xi)), \qq\bF_{\Theta,--}(\xi)=\bF_\Theta(w_{PS}(\xi)).\]
We aim to obtain the asymptotic expansion of $F(\xi) := \Delta(\xi) \Delta(w_P(\xi)) \Delta(w_S(\xi)) \Delta(w_{PS}(\xi))$ as $|\xi| \rightarrow \infty$, $\xi \in \ii\R$.
Below, we use that \[\lb{identification2}\bF_\Theta(w_P(\xi))|_{\cK_{++}}=\bF_\Theta(\xi)|_{\cK_{-+}},\qq \bF_\Theta(w_S(\xi))|_{\cK_{++}}=\bF_\Theta(\xi)|_{\cK_{+-}},\qq \bF_\Theta(w_{PS}(\xi))|_{\cK_{++}}=\bF_\Theta(\xi)|_{\cK_{--}}\]
and that these functions are bounded on their respective Riemann sheets. By using the asymptotic expansions of $\bF(0,\xi)$, $\bF'(0,\xi)$ on the unphysical sheets, the definition \begin{equation} \label{FThetaF}
   \bF_\Theta(\xi) = \bF'(0,\xi) + \Theta(\xi) \bF(0,\xi),\qq \bF(x,\xi)=\bF^+(x,\xi)
\end{equation}
(here we may not use the expansion for the inverse of the Weyl function, $\bM^{-1}$),
by tedious calculations we get the following asymptotic expansions 

\begin{lemma}[Riemann surface]
Let
\begin{align}
C = 8\mu_0 (G_{11}^H)^2,\qq D=8\omega^2\mu_0G_{11}^H \left(\frac12 G_{11}^Hc_0 H+G_{21}^H\right).
\end{align}
Let  $\omega>0$ fixed. Then, for $\xi \in \cK_{++}$ as $|\xi|\rightarrow\infty$,
\begin{eqnarray}
   \Delta(\xi)&=&\ii \xi^2 A + \ii \xi B + o(\xi),
\lb{Exp-1}\\
   \Delta(w_{PS}(\xi))&=&\ii \xi^2 A - \ii \xi B + o(\xi),
\\
   \Delta(w_P(\xi))&=& \ii\xi^4 C + \ii \xi^3 D + o(\xi^3),
\\
   \Delta(w_S(\xi))&=&\ii \xi^4 C - \ii \xi^3 D + o(\xi^3).
\lb{Exp-0}
\end{eqnarray}
\end{lemma}

Supposing better regularity of Lam{\'e} parameters, $\mu,\lambda\in C^N$ for $N$ large enough, we can prove existence of expansions
\begin{eqnarray}
   \Delta(\xi)&=&\sum_{j=0}^M c_{1,j} \xi^{2-j} + {\mathcal O}(\xi^{1-M}),
\\
   \Delta(w_{PS}(\xi))&=&\sum_{j=0}^M c_{2,j} \xi^{2-j} + {\mathcal O}(\xi^{1-M}),
\lb{Exp1}\\
   \Delta(w_P(\xi))&=&\sum_{j=0}^M c_{3,j} \xi^{4-j} + {\mathcal O}(\xi^{3-M}),
\\
   \Delta(w_S(\xi))&=&\sum_{j=0}^M c_{4,j} \xi^{4-j} + {\mathcal O}(\xi^{3-M}),
\lb{Exp2}
\end{eqnarray}
where $M$ depends on $N$.

\subsection{Cartwright class of entire functions}\label{ss-C}

In this subsection, we summarize some well-known facts from the theory of
entire functions. We follow \cite{IantchenkoKorotyaev2014a, IantchenkoKorotyaev2014b} and originally \cite{Koosis1988} and
skip the proofs given there. We write $z = x + \ii y$. An entire function $f = f(z) $ is in Cartwright class ${Cart}_{\rho_+,\rho_-},$   if $f$ is of exponential type, that is, there exists a constant $A > 0$ such that
\begin{equation}\label{first}
|f(z)| \leq {\rm const}\ e^{A|z|} 
\end{equation} 
for $z \in \C$,
\begin{equation}\label{second} \rho_{\pm}(f) \equiv \lim \sup_{y\to \infty} \frac{\ln|f(\pm \ii y)|}{y}
\end{equation}
($\rho_{\pm}$ is called exact ``type'') and
\begin{equation}\label{third}
\int_{\mathbb{R}} \frac{\log(1+ |f(x)| )dx}{1 + x^2} < \infty.
\end{equation}
Assume now that $f$ belongs to a Cartwright class and denote by
$(z_n)_{n=1}^\infty$ the sequence of its zeros $\neq 0$ (counted
with multiplicity), so arranged that $0<|z_1|\leq|z_2|\leq\ldots.$
Then we have the Hadamard factorization
\begin{equation}\label{Hadamard}f(z)= z^{\ell} f_0\lim_{R\rightarrow\infty}\prod_{|z_n|\leq R}\left(1-\frac{z}{z_{n}}\right),\qq f_0 =\frac{f^{(\ell)}(0)}{\ell !}\end{equation}
 for some integer $m,$ where
the product converges uniformly in every bounded disc and
\begin{equation}
\label{sumcond}
\sum \frac{|\Im z_n |}{|z_n|^2} <\infty.
\end{equation}
We will only need the following sub-class of ${ Cart}_{\rho_+,\rho_-},$ satisfying $ \rho_+(f) =\rho_-(f)=:\rho ,$ which we will denote by
${ Cart}_\rho$. 

We denote the
number of zeros of a function $f$ having modulus  $\leq r$ by $\cN
(r,f)$, each zero being counted according to its multiplicity. We denote by $\cN_+
(r,f)$ (or $\cN_-
(r,f)$) the
number of zeros of function $f$  counted in $\cN (r,f)$ with non-negative (negative) imaginary part  having modulus  $\leq r$, each zero being counted according to its multiplicity. We need the following
 well known result (see  \cite{Koosis1988}, p. 69).

\begin{theorem}[Levinson]\label{th-Levinson} Let the function $f$ belong to the Cartwright class ${ Cart}_\rho$  for some $\rho >0.$
Then
$$
  \cN_+(r,f)= \cN_-(r,f) =\frac{\rho\, r}{ \pi }(1+o(1)),\ \ \ r\to \iy .
$$
For each $\delta >0$ the number of zeros of $f$ with modulus $\leq r$
lying outside both of the two sectors $|\arg z | , |\arg z -\pi |<\delta$
is $o(r)$ for large $r$.
\end{theorem}

\begin{lemma} \label{L3.2}
Let $f \in { Cart}_\rho$, $\rho > 0$. Assume that for some $p \geq 0$ there exist  a rational function $G_{m,p}(z) = z^m + z^{m-1} c_1 + \ldots + c_m + \sum_{l = 1}^p c_{m+l} z^{-l}$ for some  $m \in \mathbb{N}$ and a constant $C_p$ such that
\begin{equation} \label{3.10}
   C_p = \sup_{x \in \R} |x^{p+1} (f(x) - G_{m,p}(x))| < \iy .
\end{equation}
Then for each zero $z_n \in \C_-$, $n = 1,2,\ldots$, of $f$ the following estimate holds true
\begin{equation} \label{3.11}
   |G_{m,p}(z_n))| \leq C_p |z_n|^{-p-1} e^{-\rho y_n} ,\qq y_n = \Im z_n .
\end{equation}
\end{lemma}

\begin{corollary} \label{Cor3.3.} 
Let $f \in{ Cart}_\rho$, $\rho > 0$. Let $z_n$, $n = 1,2,\ldots$, be the
zeros of $f$. 
\begin{itemize}
\item[(i)] Assume that $C_0 = \sup_{x \in \R} |x (f(x) - x^m -
x^{m-1} c_1 - \ldots - c_m)| < \iy$. Then each zero $z_n\in\C_-$, $n =
1,2,\ldots,$ satisfies
\begin{equation} \label{3.14a}
   |z_n (z_n^m + z_n^{m-1} c_1 + \ldots + c_m)| \leq C_0 e^{-\rho y_n} .
\end{equation}
\item[(ii)]  Assume that $C_1 = \sup_{x \in \R} |x^2 (f(x) - x^m -
x^{m-1} c_1 - \ldots - c_m-c_{m+1}x^{-1})| < \iy$. Then each zero $z_n\in\C_-$, $n =
1,2,\ldots,$ satisfies
\begin{equation} \label{3.14b}
   |z_n^2 (z_n^m + z_n^{m-1} c_1 + \ldots + c_m+c_{m+1} z_n^{-1})| \leq C_1 e^{-\rho y_n} .
\end{equation}
\end{itemize}
\end{corollary}

\begin{remark}
In the next section, we will prove results that imply that the (determinant of) the Jost function of the Rayleigh problem is in a Cartwright class. A function in such a class can be reconstructed from its zeros via the Hadamard factorization formula (\ref{Hadamard}), as was shown in the scalar Schr{\"o}dinger  case in \cite{Korotyaev2004}. In a forthcoming paper, we will analyze the inverse wavenumber resonances problem using this fact.
\end{remark}

\subsection{Analytic properties of the relevant entire functions and proof of the main results}\label{ss-A}

We now present the analytic properties of all the relevant entire functions in the analysis. Let \[\lb{functionF}F(\xi) = \Delta(\xi)\Delta(w_S(\xi))\Delta(w_{P}(\xi))\Delta(w_{PS}(\xi)),\,\,  \omega>0.\] We note that $F$ is entire on $\C$ with zeros which are projection on $\C$ of all wave number resonances. In this subsection, to adhere to standard notation, we will identify $\xi$ with $-\ii z$. Clearly, $z \in \C_+$ corresponds to $\{\xi,\,\,\Re\xi\geq 0\}$. The main result of this subsection is

\begin{theorem}[Cartwright property]\label{th-Cp}
The components of the the boundary matrix, the entire functions $\gamma_j(-\ii z)$, $j=1,\ldots,8$, are in Cartwright class ${Cart}_{2H}$. The entire functions $d_1(-\ii z)$, $d_2(-\ii z),$ $d_3(-\ii z),$ $d_4(-\ii z)$ and $\cS(-\ii z) = \cP(-\ii z)$ are in Cartwright class ${Cart}_{4H}$, and the entire function $F(-\ii z)$ is in Cartwright class ${Cart}_{8H}$.
\end{theorem}

\begin{proof}
With (\ref{BFTheta}), it is sufficient to consider the asymptotics of the entries of $\bF_\Theta(\xi)$ as $|\xi| \rightarrow \infty$ for $\xi \in \R$ and $\xi \in \ii \R$. Property (\ref{first}) and the exact type $2H$ (property (\ref{second}) of the components of vector functions $\gamma_j$, $j=1,\ldots,8$, follow from the bounds and  the asymptotics of $\gamma$ as $\xi \rightarrow \pm\infty$ as $\xi \in \R$ which follow from Theorem~\ref{th-Jost-solutions}. Moreover, from (\ref{FThetaF}) we get the limits of the entries of $\mB$ as $\xi \rightarrow \pm\infty$: $\rho_+ = 0$ on the physical sheet ($\xi \rightarrow +\infty$), $\rho_- = 2H$ on the unphysical sheet ($\xi \rightarrow -\infty$). Now, taking the determinant of $\mB$, we get $\rho_+(\Delta (\xi)) = 0$, $\rho_-(\Delta (\xi)) = 4H$ and $\rho_+(\Delta (w_S(\xi)) = 4H$, $\rho_-(\Delta (w_S(\xi)) = 0$. As from $\Delta = d_1+q_Pd_2+q_Sd_3+q_Pq_Sd_4$ and the definitions of the maps $w_\bullet$ it follows that
$$\begin{array}{rl} 4d_1=&\Delta(\xi)+\Delta(w_S(\xi))+\Delta(w_P(\xi))+\Delta(w_{PS}(\xi)),\\
4q_Pd_2= &\Delta(\xi)+\Delta(w_S(\xi))-\Delta(w_P(\xi))-\Delta(w_{PS}(\xi)),\\
4q_Sd_3=&\Delta(\xi)-\Delta(w_S(\xi))+\Delta(w_P(\xi))-\Delta(w_{PS}(\xi)),\\
4q_Pq_Sd_4=&\Delta(\xi)-\Delta(w_S(\xi))-\Delta(w_P(\xi))+\Delta(w_{PS}(\xi)),
\end{array}$$
we find that $d_1,\ldots,d_4$ have exact type $4H$. As
\begin{align*}
   F = (d_1^2-q_P^2d_2^2-q_S^2d_3^2+q_P^2q_S^2   d_4^2)^2-4q_P^2q_S^2(d_1d_4-d_2d_3)^2 ,
\end{align*}
$F$ must have exact type $8H$. Moreover, as $\cS^2 = 4 (d_1 d_4-d_2 d_3)$, we find that $\cS^2 = \cP^2$ have exact type $8H$, so that $\cS = -\cP$ have exact type $4H$. 

To verify property (\ref{third}) for $F$, we use (\ref{Exp-1})-(\ref{Exp-0}), which imply that as $|\xi| \rightarrow \infty$, $\xi \in \ii\R$, 
\begin{align*}F=A^2C^2
\xi^{12}
+o(\xi^{12}).
\end{align*}
To verify property (\ref{third}) for the $\gamma_j$, $j=1,\ldots,8$, and for $d_1,\ldots,d_4$, we use the asymptotics of both $\bF^\pm(Z,\xi)$ and their derivatives with respect to $\xi$ implying at most polynomial growth as $|\xi| \rightarrow \infty$, $\xi \in \ii\R$, due to properties (\ref{symmetry-1}--\ref{symmetry-2}) and identifications (\ref{identification1}--\ref{identification2}).
\end{proof}

Finally, we show how the results on the distribution of wavenumber resonances in the complex plane, presented in the introduction, are obtained. Using that wavenumber resonances are identified with the zeros of $F$, we get these directly from the Cartwright character of function $F(-\ii z)$ proved in Theorem \ref{th-Cp} together with general  properties of the zeros of an entire function from a Cartwright class summarized in Section~\ref{ss-C}.

First, using  (\ref{sumcond}) we obtain (\ref{xisumcond}). Second, from Theorems~\ref {th-Levinson} and \ref{th-Cp} we immediately find that
\[\label{Weylas+-}
   \cN_+(r,F) = \cN_-(r,F)
   =\frac{8 H \, r}{\pi} (1 + o(1)),\qq r\rightarrow \infty ;
\] 
moreover we deduce (\ref{Weylas}). Third, if $\mu, \lambda \in C^N$ for some $N$ large enough, we use (\ref{Exp1})-(\ref{Exp2}) to conclude that, as $|\xi| \rightarrow \infty$, $\xi \in \ii\R$ and $\xi \in \cK_{++}$, \[\lb{ExpF}
   F(\xi) = A^2 C^2 (\xi^{12} - c_2 \xi^{10} + c_4 \xi^{8} - c_6 \xi^{6}
      + c_8 \xi^4 - c_{10} \xi^2 + c_{12}-c_{14}\xi^{-2}) + o(\xi^{-2}).
\]
Applying Corollary \ref{Cor3.3.} to $f = F$ with $z = \ii \xi$, $m = 12$ and using that $F$ is even, imply that each wavenumber resonance $\xi_n$, $n = 1,2,\ldots$, satisfies (\ref{3.14abisbis}). For the zeros of $\cS$, $\cP$, that is, the ``no-mode conversions wavenumbers'', we obtain similar observations. Here, we note that
$$
   \cP(\xi) = -\cS(\xi) = \ii C (-\ii \xi^3 - c_1 \xi^2 + c_2 \ii \xi
   + c_3) + o(1)
$$
as $|\xi| \rightarrow \infty$, $\xi \in \ii\R$ and $\xi \in \C$.

\section{Discussion}

We analyzed wavenumber resonances of the Rayleigh system on the appropriate Riemann surface. To this end, we introduced and studied Jost solutions, the associated boundary matrix and their analytic properties both on the physical and nonphysical sheets of the Riemann surface. Here, conjugation properties between sheets appeared to be important. The determinant of the boundary matrix is the Rayleigh determinant. We developed a representation for the Rayleigh resolvent admitting an analytic continuation, the poles of which coincide with the wavenumber resonances. Using the mentioned analytic properties, we then obtained an estimate for the distribution of the wavenumber resonances on the mentioned physical sheet as well as an asymptotic estimate for their counting function.

With the results presented here, we prepared the analysis of the associated inverse problems that will be developed in a forthcoming paper. We conjecture that the wavenumber resonances together with the no-mode-conversion wavenumbers determine the reflection matrix and that the reflection matrix determines the boundary matrix making use of the Hadamard factorizations for $F$ and $\cS = -\cP$. By (\ref{BFTheta}), the boundary matrix is directly related to the Jost function and this function determines the \textit{P}- and \textit{S}-wave speeds as functions of depth (the boundary normal coordinate) as we proved in a previous paper \cite{MdHoopIantchenko2022}.

\section*{Acknowledgments}

MVdH was supported by the Simons Foundation under the MATH $+$ X program, the National Science Foundation under grant DMS-1815143, and the corporate members of the Geo-Mathematical Imaging Group at Rice University.

\appendix

\section{Markushevich transform to two adjoint matrix Sturm-Liouville
  problems}
\label{sec:Mark}

We perform an analogue of the calibration transform on the Rayleigh system to obtain a matrix-valued (essentially non-diagonalizable) Sturm-Liouville problem. We follow \cite{ArgatovIantchenko2019}. Based on the Pekeris substitution \cite{Pekeris1934}, it was shown by Markushevich \cite{Markushevich1989, Markushevich1994} that the boundary value problem (\ref{Rayleighsystem}) with the Neumann boundary conditions (\ref{Rayleighboundary1})-(\ref{Rayleighboundary2}) can be reduced to two matrix Sturm-Liouville problems with mutually transposed potentials and boundary conditions. Here, we briefly review the transformations for arbitrary boundary values, $\chi_1, \chi_2$ for the traction instead of zero on the right-hand side of (\ref{Rayleighboundary1})-(\ref{Rayleighboundary2}), following very closely  our earlier work \cite{MdHoopIantchenko2022}. For conciseness of notation while suppressing the coordinate dual to $\xi$, in the remainder of the analysis, we use a ${}'$ to denote differentiation with respect to $x$.

Let $G$ be a $2 \times 2$-matrix solving the Cauchy problem,
\begin{equation}
   G' = \frac{1}{2} L G ,\quad G(0) = I_2 ,
\label{1rw(2.1)}
\end{equation}
where $I_2$ is the unit matrix, and 
\begin{equation}
   L = \left(\begin{array}{cc}
              0 & -d \\
             -c &  0 \end{array}\right)\quad\text{with}\quad
   c = \frac{1}{g_0} \frac{\mu (\lambda + \mu)}{(\lambda + 2 \mu)} ,\
   d =- 2 g_0 \biggl(\frac{1}{\mu}\biggr)'' .
\label{1rw(2.2)}
\end{equation}
We have $\det G(x) = 1,$ cf.~\cite{Markushevich1986}. 
We adopt the notation of Markushevich \cite{Markushevich1989}, where
$g_0$ stands for an arbitrary positive constant. It is convenient to
put $g_0 = \mu_0$, which we do from now onwards.
% {\color{blue} If $d=-2\mu_0\biggl(\frac{1}{\mu} \biggr)''=0,$ solution
%   to (\ref{1rw(2.1)}) is given by
% $$G(x)=e^{\frac12\int_0^x L(t)dt}=\exp \left(\begin{array}{cc}
% 0& 0 \\
% -\frac12\int_0^x c(t)dt & 0
% \end{array}\right)=\left(\begin{array}{cc}
% 1& 0 \\
% -\frac12\int_0^x c(t)dt & 1
% \end{array}\right).$$
% If $d\not =0$ we cannot write solution in $\exp$ form, due to
% non-commuting $\int_0^x L(t)$ and $L(x).}
By the substitution ($x \in [0,\infty)$)
\begin{equation} \label{MPtr}
   \gM^{-1}(F) = \left(\begin{array}{c}
         w_1 \\
         w_2 \end{array}\right)
\end{equation}
with
\begin{equation}
  \gM^{-1}
  = \left(\begin{array}{cc}
          \displaystyle \dede{}{x} & 1 \\
          -\xi & 0 \end{array}\right)
    \left(\begin{array}{cc}
    \displaystyle \frac{\mu_0}{\mu} & 0 \\
    0 & \displaystyle \frac{\mu}{\lambda+2\mu} \end{array}\right)
    \bigl(G^{\rm T}\bigr)^{-1}
\label{1rw(2.10)new}
\end{equation}
signifying the \textit{inverse} Markushevich transform, the boundary value problem
(\ref{Rayleighsystem}) 
with the Neumann boundary conditions (\ref{Rayleighboundary1})-(\ref{Rayleighboundary2})
 reduces to the matrix
Sturm-Liouville form
\begin{eqnarray}
   F'' - \xi^2 F &=& Q F ,\quad\ \ \ x \in (0,\infty) ,
\label{1rw(2.11)new} \\
   F'\ + \Theta F &=& \bigl(D^{\rm a}\bigr)^{-1} \chi ,\quad
   x = 0 .
\label{1rw(2.12)new}
\end{eqnarray}
Here, $\Theta = \Theta(\xi) = (D^{\rm a}(\xi))^{-1} C^{\rm a}(\xi)$
with
\begin{multline}
   D^{\rm a}(\xi) = \left(\begin{array}{cc}
         \displaystyle -2 \mu_0 \frac{\mu'(0)}{\mu(0)} & \mu(0) \\
         -2 \mu_0 \xi & 0 \end{array}\right) ,
\\   
   C^{\rm a}(\xi) = \left(\begin{array}{cc}
     \displaystyle \mu_0 \biggl(2 \xi^2 - \frac{\omega^2}{\mu(0)}
                       + \frac{\mu''(0)}{\mu(0)}\biggr) &
     \displaystyle -\frac{\mu'(0) \mu(0)}{\lambda(0) + 2 \mu(0)} \\
     \displaystyle 2 \mu_0 \xi \frac{\mu'(0)}{\mu(0)} &
     \displaystyle -\xi \frac{\mu^2(0)}{\lambda(0) + 2 \mu(0)}
     \end{array}\right) ,
\label{1rw(2.14)new}
\end{multline}
so that 
\begin{equation} \label{Da}
  \left(D^{\rm a}(\xi)\right)^{-1}
  = \frac{1}{2 \mu_0 \mu \xi}
    \left(\begin{array}{cc} 0 & -\mu(0) \\
    2 \mu_0 \xi & \displaystyle -2 \mu_0 \frac{\mu'(0)}{\mu(0)} 
    \end{array}\right)
\end{equation}
and
\begin{align}
   \Theta(\xi) &= \left(\begin{array}{cc}
   \displaystyle -\frac{\mu'(0)}{\mu(0)} &
   \displaystyle \frac{1}{2 \mu_0}
        \frac{\mu^2(0)}{(\lambda(0)+2\mu(0))} \\
   \displaystyle \frac{\mu_0}{\mu(0)}
        \biggl(2 \xi^2 - \frac{\omega^2}{\mu(0)}
   - \mu(0) \biggl(\frac{1}{\mu}\biggr)''(0)\biggr) & 0
   \end{array}\right)
\label{1rw(2.15)new} \\
\nonumber \\
  &=: \ma -\theta_3 & \theta_2 \\ \\
      \displaystyle 2 \frac{\mu_0}{\mu(0)} \xi^2 - \theta_1 & 0 \am .
\nonumber
\end{align}
Furthermore, $Q$ is the matrix-valued potential given by
\begin{equation} \label{1rw(2.8)new}
   Q = (G^{-1} B G)^{\rm T} ,\quad B = B_1 + \omega^2 B_2 ,
\end{equation}
with
\begin{equation}
   B_1 = \left(\begin{array}{cc}
   \displaystyle -\frac{1}{2}\biggl(\frac{1}{\mu}\biggr)''
   \frac{\mu (\lambda + \mu)}{\lambda + 2 \mu}
   + \frac{\mu''}{\mu} &
   \displaystyle \mu_0\biggl(2\frac{\mu'}{\mu}
   \biggl(\frac{1}{\mu}\biggr)''
       + \biggl(\frac{1}{\mu}\biggr)'''\biggr)
   \\
   \displaystyle \frac{1}{\mu_0} \biggl(\frac{\lambda' \mu^2
   + \mu' \lambda (\lambda + \mu)}{(\lambda + 2 \mu)^2}
   - \frac{1}{2} \biggl(\frac{\mu (\lambda + \mu)}{
       \lambda + 2 \mu}\biggr)'\biggr) &
   \displaystyle \frac{1}{2} \biggl(\frac{1}{\mu}\biggr)''
     \frac{(\lambda - \mu)}{\lambda + 2 \mu}
   \end{array}\right) ,
\label{1rw(2.9a)new}
\end{equation}
\begin{equation}
   B_2 = \left(\begin{array}{cc}
     \displaystyle -\frac{1}{\mu} &
     \displaystyle \mu_0 \biggl(\frac{1}{\mu^2}\biggr)'
   \\
   \displaystyle 0 & \displaystyle -\frac{1}{\lambda + 2 \mu}
   \end{array}\right) .
\label{1rw(2.9b)new} 
\end{equation}
We note that the potential is not a symmetric matrix, that is, $Q \neq
Q^{\rm T}$.

By the adjoint substitution 
\begin{equation}
   \left(\gM^{\rm a}\right)^{-1}(F^{\rm a})
   = \left(\begin{array}{c}
         w_1 \\
         w_2 \end{array}\right)
\end{equation}
with
\begin{equation}
  \left(\gM^{\rm a}\right)^{-1}
  = \left(\begin{array}{cc}
    0 & -\xi \\ 1 & \displaystyle \dede{}{x}
    \end{array}\right)
    \left(\begin{array}{cc}
      1 & \displaystyle -2 \mu_0 \biggl(\frac{1}{\mu}\biggr)'
      \\
      0 & \displaystyle \frac{\mu_0}{\mu}\end{array}\right)
    G ,
\label{1rw(2.3)new}
\end{equation}
the boundary value problem
(\ref{Rayleighsystem}) 
with the Neumann boundary conditions (\ref{Rayleighboundary1})-(\ref{Rayleighboundary2})
 transforms to the matrix
Sturm-Liouville form
\begin{eqnarray}
  (F^{\rm a})'' - \xi^2 {F^{\rm a}} &=& Q^{\rm a} {F^{\rm a}} ,
  \quad x \in (0,\infty) ,
\label{1rw(2.4)new} \\
  (F^{\rm a})' + \Theta^{\rm a} {F^{\rm a}} &=&
           D^{-1} \chi ,\quad x = 0 .
\label{1rw(2.5)new}
\end{eqnarray}
Here,
\begin{equation}
    Q^{\rm a} =  Q^{\rm T} ,\quad
    \Theta^{\rm a} = \Theta^{\rm T} = \Theta^{\rm T}(\xi)
    = D^{-1}(\xi) C(\xi)
\end{equation}
is a $2 \times 2$-matrix with $D(\xi)$ and $C(\xi)$ being the
matrices,
\begin{multline}
   D(\xi) = \left(\begin{array}{cc}
       0 & -2 \xi \mu_0 \\[0.25cm]
       \mu(0) & 0 \end{array}\right) ,
\\
   C(\xi) = \left(\begin{array}{cc}
   \displaystyle -\xi \frac{\mu^2(0)}{(\lambda(0) + 2 \mu(0))} & 0
   \\
   -\mu'(0) & \displaystyle \frac{\mu_0}{\mu(0)}
   \biggl(2 \mu(0) \xi^2 - \omega^2 - 2 \frac{(\mu'(0))^2}{\mu(0)}
   + \mu''(0)\biggr) \end{array}\right) .
\label{1rw(2.7)new}
\end{multline}

\subsection*{Homogeneous half space, $x \in (H,\infty)$}

In components, (\ref{1rw(2.1)}) has the form 
\begin{equation}
   G_{11}' = -\frac{d}{2} G_{21} ,\quad
   G_{12}' = -\frac{d}{2} G_{22} ,\quad
   G_{21}' = -\frac{c}{2} G_{11} ,\quad
   G_{22}' = -\frac{c}{2} G_{12} ,
\label{1rw(3.1)new}
\end{equation}
in which, in view of (\ref{1rw(2.2)}), the coefficient $d$ is zero if
$\mu$ is constant. 
We consider the (homogeneous) half space $x \in (H,\infty)$ and write
\begin{equation}
   G_{11}(H) = G_{11}^H ,\quad
   G_{12}(H) = G_{12}^H ,\quad
   G_{21}(H) = G_{21}^H ,\quad
   G_{22}(H) = G_{22}^H .
\label{1rw(3.6)new}
\end{equation}
Then the matrix function, $G$, inside $x \in (H,\infty)$ can be
determined from the Cauchy problem
\begin{equation}
   G' = -\frac{c_0}{2} \left(\begin{array}{cc}
            0 & 0 \\[0.25cm] G_{11} & G_{12} \end{array}\right) ,\quad
   G(H) = \left(\begin{array}{cc}
            G_{11}^H & G_{12}^H \\[0.25cm]
            G_{21}^H & G_{22}^H \end{array}\right) ,
\label{1rw(3.7)new}
\end{equation}
in which
\begin{equation}
   c_0 = \frac{\lambda_0 + \mu_0}{\lambda_0 + 2 \mu_0} .
\label{1rw(3.8)new}
\end{equation}
The solution is
\begin{equation}
\begin{array}{c}
   \displaystyle \phantom{{}_{\Bigr)}}
   G_{11}(x) = G_{11}^H ,\quad G_{12}(x) = G_{12}^H ,
   \phantom{{}_{\Bigr)}} \\
   \displaystyle 
   G_{21}(x) = -\frac{c_0}{2} G_{11}^H(x - H) + G_{21}^H ,\quad
   G_{22}(x) = -\frac{c_0}{2} G_{12}^H(x - H) + G_{22}^H .
\end{array}
\label{1rw(3.9)new}
\end{equation}
As $\det G(x) = 1$ (see \cite{Markushevich1986, Markushevich1989}), the inverse matrix follows to be
\begin{equation}
   G^{-1} = \left(\begin{array}{rr}
                G_{22} & -G_{12} \\[0.25cm]
               -G_{21} & G_{11} \end{array}\right) .
\label{1rw(3.10)new}
\end{equation}
Thus, in the homogeneous elastic half space, $x \in (H,\infty)$,
according to (\ref{1rw(2.8)new})-(\ref{1rw(2.9b)new}) and
(\ref{1rw(3.10)new}), we have
\begin{equation} \label{1rw(3.11)new}
   Q = \omega^2 \left(\begin{array}{cc}
   \displaystyle \frac{G_{12} G_{21}}{\lambda_0 + 2 \mu_0}
   - \frac{G_{11}G_{22}}{\mu_0} &
   \displaystyle G_{11} G_{21} \frac{c_0}{\mu_0}
   \\[0.25cm]
   \displaystyle -G_{12} G_{22} \frac{c_0}{\mu_0} &
   \displaystyle \frac{G_{12} G_{21}}{\mu_0}
                 - \frac{G_{11} G_{22}}{\lambda_0 + 2 \mu_0}
   \end{array}\right) ,
\end{equation}
where the components of the transformation matrix $G$ are given by
(\ref{1rw(3.9)new}). It is of interest to observe that if $G_{12}^H
\not= 0$, then all components of the potential matrix, $Q$, will have
linear growth as $x \to \infty$.

In the main text, we denote $Q(x)$ for $x \geq H$ by $Q_0(x)$.
Using, again, that $\det G(x) = 1$, we obtain
\begin{align}
  Q_0(x) =&
  \omega^2 \ma \displaystyle -\frac{1}{\mu_0} & 0 \\
  0 & \displaystyle -\frac{1}{\lambda_0 + 2 \mu_0} \am
  + \omega^2 \frac{\lambda_0 + \mu_0}{\mu_0 (\lambda_0 + 2 \mu_0)}
  \ma -G_{12} G_{21} & G_{21}G_{11} \\[0.25cm]
      -G_{12} G_{22} & G_{12}G_{21} \am
\label{Q0} \\[0.25cm]
  =& \omega^2 \ma \displaystyle - \frac{1}{\mu_0} & 0 \\
  0 & \displaystyle -\frac{1}{\lambda_0 + 2 \mu_0} \am
\nonumber\\[0.25cm] &\qquad
  + \omega^2\frac{c_0}{\mu_0}
  \ma \displaystyle
  -G_{12}^H \left[-\frac{c_0}{2}G_{11}^H(x - H) + G_{21}^H\right] &
  \displaystyle 
  G_{11}^H \left[-\frac{c_0}{2} G_{11}^H(x - H) + G_{21}^H\right]
  \\[0.25cm]
  \displaystyle
  -G_{12}^H \left[-\frac{c_0}{2} G_{12}^H(x - H) + G_{22}^H\right] &
  \displaystyle 
  G_{12}^H \left[-\frac{c_0}{2} G_{11}^H(x - H) + G_{21}^H\right]
  \am .
\nonumber
\end{align}
We extend $Q_0 = Q_0(x)$ to $x \in (0,H]$ linear in $x$, and refer to
it as the background potential. Then we introduce the perturbation
potential $$V(x) = Q(x) - Q_0(x)$$ so that $V(x) = 0$ for $x \geq H$.

\bibliographystyle{ws-m3as}
%\bibliography{references}

\end{document}